\documentclass[10pt, a4paper]{amsart}

\usepackage[utf8]{inputenc}
\usepackage[english]{babel}

\usepackage{amsthm}
\usepackage{amstext}
\usepackage{amsmath}
\usepackage{amsfonts}
\usepackage{amssymb}
\usepackage{mathtools}

\usepackage{tikz}
\usepackage{graphics}
\usepackage{wrapfig}
\usetikzlibrary{matrix,arrows,decorations.pathmorphing}

\usepackage[a4paper]{geometry}
\usepackage{tikz-cd}
\usepackage{epigraph}

\usepackage{hyperref}
\usepackage{cite}

\usepackage{enumerate}
\usepackage{enumitem}

\author{Matteo Tamiozzo}
\title{On the Bloch-Kato conjecture for Hilbert modular forms}
\date{\today}

\swapnumbers
\newtheorem{teo}[subsection]{Theorem}     
\theoremstyle{plain}                    
\newtheorem{prop}[subsection]{Proposition}    
\newtheorem{corol}[subsection]{Corollary}     
\newtheorem{lem}[subsection]{Lemma}         
\theoremstyle{definition}               
\newtheorem{defin}[subsection]{Definition}
\theoremstyle{remark}                   
\newtheorem{rem}[subsection]{Remark}      
\newtheorem{notat}[subsection]{Notation}
\newtheorem{ass}[subsection]{Assumption}

\newcommand{\R}{\mathbf{R}}     
\newcommand{\C}{\mathbf{C}}     
\newcommand{\Q}{\mathbf{Q}}      
\newcommand{\Z}{\mathbf{Z}}    
\newcommand{\af}{\mathbf{A}_f}     

\newcommand{\pl}{\mathfrak{l}}
\newcommand{\p}{\mathfrak{p}}
\newcommand{\q}{\mathfrak{q}}
\newcommand{\n}{\mathfrak{n}}
\newcommand{\np}{\mathfrak{n}^+}
\newcommand{\nm}{\mathfrak{n}^-}
\newcommand{\OF}{\mathcal{O}_F}

\newcommand{\Tn}{\mathbf{T}_\mathfrak{n}}
\newcommand{\Op}{\mathcal{O}_\p}

\newcommand{\lOp}{l_{\Op}}

\newcommand{\Addr}{{\bigskip
\footnotesize
\textsc{Department of Mathematics, Imperial College London, London SW7 2AZ, UK}

\textit{Email address:} \texttt{m.tamiozzo@imperial.ac.uk}
}}

\numberwithin{equation}{subsection}

\begin{document}

\date{}

\begin{abstract}
The aim of this paper is to prove inequalities towards instances of the Bloch-Kato conjecture for Hilbert modular forms of parallel weight two, when the order of vanishing of the $L$-function at the central point is zero or one. We achieve this implementing an inductive Euler system argument which relies on explicit reciprocity laws for cohomology classes constructed using congruences of automorphic forms and special points on several Shimura curves.
\end{abstract}

\maketitle

\tableofcontents

\section{Introduction}

\subsection{} Let $F$ be a totally real field with ring of integers $\mathcal{O}_F$, $\n\subset \mathcal{O}_F$ an ideal and $f \in S(\n)$ a Hilbert newform of parallel weight two and level $U_1(\n)$ with trivial central character. Let $E$ be the number field generated by the Hecke eigenvalues of $f$. One can attach to $f$ a compatible system of (self-dual) Galois representations indexed by finite places of $E$, coming in most cases from a motive $M$ over $F$ with coefficients in $E$  (\hspace{1sp}\cite{blro93}) whose $L$-function coincides with the automorphic $L$-function $L(f, s)$. The conjectures of Bloch and Kato \cite{bk90} - reformulated and extended by Fontaine and Perrin-Riou \cite{fon92}, \cite{fpr94}- predict that the order of vanishing of $L(f, s)$ at the central point $s=1$ should be equal to the dimension of the Selmer group of (the \'etale realisations of) $M$, and express the first non zero term in the Taylor expansion of $L(f, s)$ at $s=1$ in terms of arithmetic invariants of $M$.

\subsection{} The aim of this paper is to study instances of these conjectures for the base change of $M$ to a $CM$ extension $K/F$, when the order of vanishing of the relevant $L$-function is at most one. In this case we prove, under suitable assumptions, inequalities towards the special value formulas predicted by Bloch-Kato. Furthermore we are able to provide a criterion under which our inequalities can actually be shown to be equalities.

In order to state our main result we need to introduce some more notation: fix a place $\p$ of $E$ lying above a rational prime $p$; let $E_\p$ be the completion of $E$ at $\p$ and let $\Op$ be the ring of integers of $E_\p$. Let $\rho:Gal(\bar{F}/F)\rightarrow Aut(V(f))$ the $\p$-adic Galois representation attached to $f$. Choose a self-dual $Gal(\bar{F}/F)$-stable $\Op$-lattice $T(f)\subset V(f)$; set $A(f)=V(f)/T(f)$ and let $\bar{\rho}: Gal(\bar{F}/F)\rightarrow Aut(T(f)/\p)$ be the residual Galois representation attached to $f$.

\subsection{} Assume that $\n$ is squarefree and all its prime factors are inert in $K$. The sign of the functional equation of $L(f_K, s)$ equals 1 (resp. -1) if the number of prime ideals dividing $\n$ has the same (resp. opposite) parity as the degree $[F:\Q]$. In the first case, called the \emph{definite case}, one can define the \emph{algebraic part} of the special value $L(f_K, 1)$, denoted by $L^{alg}(f_K, 1)$ (see Remark \ref{fixtransf}); our result relates its $\p$-adic valuation $v_\p(L^{alg}(f_K, 1))$ to the length of the Selmer group $Sel(K, A(f))$. In the second case, called the \emph{indefinite case}, the representation $T(f)$ can be realised as a quotient of the $p$-adic Tate module of a suitable Shimura curve, and one can use points with $CM$ by $K$ on the curve to construct a Selmer class $c \in Sel(K, T(f))$ which is non zero if and only if $L'(f_K, 1)\neq 0$ (see \ref{defindefclass}); if this is the case let $v_\p(c)=\max\{k \geq 0 \mid c \in \p^k Sel(K, T(f))\}$. We will relate  $v_\p(c)$ to the length of the quotient of $Sel(K, A(f))$ by its divisible part.

\subsection{} We will work throughout the text with a certain class of automorphic forms modulo (powers of) $\p$, which we call \emph{admissible automorphic forms} (see Definition \ref{defadmeig}). Given such an automorphic form $h$ we will consider the Selmer group $Sel_{(\mathfrak{D}_h/\n)}(K, T_1(f))$ and the ``algebraic part of the special value'' $a(h) \in \Op/\p^n$ defined in (\ref{changesel}, \ref{defadmeig}). We can now state our main result:

\begin{teo}(cf. Theorems \ref{thbkdef}, \ref{mainmodnindef})\label{mainintro}
The notation being as above, assume that
\begin{enumerate}
\item the level $\n$ of $f$, the discriminant $disc(K/F)$ and the prime $p$ below $\p$ are coprime to each other. Moreover $p>3$ is unramified in $F$, and $\n$ is squarefree and all its factors are inert in $K$.
\item The image of the residual Galois representation $\bar{\rho}$ attached to $f$ contains $SL_2(\mathbf{F}_p)$.
\item For every prime $\q\mid \n$ we have $N(\q) \not \equiv -1 \pmod p$. Moreover if $N(\q) \equiv 1 \pmod p$ then $\bar{\rho}$ is ramified at $\q$.
\end{enumerate}

Then the following statements hold true:
\begin{description}
\item[Definite case] if $L(f_K, 1)\neq 0$ then $Sel(K, A(f))$ is finite and
\begin{equation*}
length _{\Op}Sel(K, A(f))\leq v_\p(L^{alg}(f_K, 1));
\end{equation*}
\item[Indefinite case] if $L'(f_K, 1)\neq 0$ then $Sel(K, A(f))$ has $\Op$-corank one and
\begin{equation*}
length _{\Op}Sel(K, A(f))/div\leq 2v_\p(c).
\end{equation*}
\end{description}

Moreover the above inequalities are equalities provided that the following implication holds true: if $h$ is an admissible automorphic form mod $\p$ and $Sel_{(\mathfrak{D}_h/\n)}(K, T_1(f))=0$ then $a(h)$ is a $\p$-adic unit.
\end{teo}

\subsection{} Various results in the spirit of the above theorem have already been proved. In particular, the implication $L(f_K, 1)\neq 0 \Rightarrow Sel(K, V(f))=0$ was studied, in different degrees of generality, by several authors (among others \cite{bd05}, \cite{lon06}, \cite{lv10}, \cite{chi17}), and established under minimal assumptions for Hilbert modular forms of parallel weight two by Nekov\'{a}\v{r} \cite{nek12}. The idea underlying all these works, dating back to the seminal work \cite{bd05}, is to use a level raising of $f$ at well chosen primes $\pl$ in order to switch from the definite to the indefinite situation, and use $CM$ points on Shimura curves, available in the latter setting, to construct cohomology classes $c(\pl)$ which, via global duality, create an obstruction to the existence of Selmer classes whenever $L(f_K, 1)$ does not vanish. This is proved by establishing an \emph{explicit reciprocity law} (the first reciprocity law in \cite{bd05}) relating the localisation of $c(\pl)$ at $\pl$ to the special value $L(f_K, 1)$.

\subsection{} In \cite{bd05} a second reciprocity law was also proved, expressing the localisation of $c(\pl)$ at suitable primes $\pl'\neq \pl$ in terms of the special value of the $L$-function of a level raising of $f$ at the two primes $\pl$ and $\pl'$. The joint use of the two reciprocity laws makes an induction process possible, which was used in \cite{bd05} in order to prove one divisibility in the anticyclotomic Iwasawa main conjecture for weight two modular forms under suitable assumptions. This work was later generalised to modular forms of higher weight in \cite{chi15} and to Hilbert modular forms of parallel weight two (resp. higher parallel weight) in \cite{lon12} (resp. \cite{wa15}). The results in these papers only apply to modular forms which are \emph{ordinary} at primes above $p$; this assumption has been removed for modular forms of weight two in \cite{darjo08}, \cite{pw11}, under the hypothesis that the prime $p$ is split in the given imaginary quadratic field $K$; the case of inert primes is investigated in the preprint \cite{blv16}. A general argument allowing to deduce one divisibility in the anticyclotomic Iwasawa main conjecture from the two reciprocity laws is presented in \cite{how06}, where a condition implying that this divisibility is an equality is also given.

The above mentioned results on the Iwasawa main conjecture can be used to deduce, for ordinary Hilbert modular forms, the inequalities in our theorem (at least in the definite case); the main point of this work consists in observing that a refinement of the Euler system argument used in \cite{bd05}, \cite{lon12} can be employed to prove \emph{directly} the sought-for inequalities. In particular, we do not need to restrict to ordinary nor split primes. In fact, by \cite[Proposition 0.1]{dim05}, given a Hilbert newform $f$ which is not a theta series and a $CM$ extension $K/F$ satisfying the requirements in $(1)$, our result applies to all but finitely many primes $\p$.

\subsection{} In the case $F=\Q$, formulas relating the length of the Selmer group of the twist of $A(f)$ by certain anticyclotomic Hecke characters to special values of $L$-functions are proved in \cite{blv16}, both in analytic rank zero and one. The authors then use these formulas to deduce the relevant anticyclotomic Iwasawa main conjecture. Our work started from the observation that part of the argument in \cite{blv16} does not rely on Iwasawa theory, and can be adapted and generalised to the setting of Theorem \ref{mainintro} (see \ref{pastworks}).

\subsection{} Let us point out that the criterion that we give to upgrade our inequalities to equalities can be seen as a $GL_2$-version of Ribet's converse of Herbrand's theorem \cite{rib76} (see Remark \ref{reverseimp}). This follows from the Skinner-Urban divisibility in the Iwasawa main conjecture \cite{su14}, \cite{wan14}, proved for ordinary Hilbert modular forms in \cite{wan15}. Using this line of thought the authors of \cite{blv16} are able to establish the full equality in the anticyclotomic Iwasawa main conjecture for elliptic curves over $\Q$, both in the ordinary and supersingular case. In our setting, this approach does not allow for the time being to prove that the inequalities in Theorem \ref{mainintro} are always equalities; it would be interesting to know whether the result that we need to achieve this, a priori weaker than its Iwasawa-theoretic counterpart, can be established under the assumptions of our theorem.

\subsection{} The hypotheses in our theorem are quite strong, and many of them are mainly needed to apply the automorphic results used to construct the Euler system (in particular a suitable multiplicity one result and Ihara's lemma; see section \ref{reclaws}). In the special case $F=\Q$ everything would work under the (weaker) assumption $CR$ made in \cite{pw11}; as already mentioned, the advantage is that hypotheses like $a_p=0$ and $p$ split in $K$, which are made in \cite{pw11} in the supersingular case for Iwasawa-theoretic reasons, are unnecessary for us.

\subsection{} Let us briefly describe how our proof works, referring the reader to section \ref{esys} for the details. In the definite case we prove the result by induction on $t(f)=v_\p(L^{alg}(f_K, 1))$. When $t(f)=0$ it is easily seen that the existence of the classes $c(\pl)$ and the first reciprocity law force the vanishing of $l(f)=length_{\Op}Sel(K, A(f))$ (Corollary \ref{ineq0}). If $t(f)>0$, using both reciprocity laws and global duality we are able to show that \emph{either} $Sel(K, A(f))=0$ \emph{or} one can produce a level raising $g$ of $f$ modulo a high enough power of $\p$ such that $t(g)<t(f)$ and $t(g)-l(g)=t(f)-l(f)$. Let us stress that the Euler system argument we give is not enough to rule out the possibility that $t(f)>0$ and $Sel(K, A(f))=0$; one needs this additional input in order to promote our inequalities to equalities. This is not surprising, as the fact that $Sel(K, A(f))\neq 0$ if $t(f)>0$ is an \emph{existence statement} for non trivial Selmer classes, which one does not expect to follow from the algebraic manipulations at the heart of our arguments.

Finally, we deal with the indefinite case by essentially reducing it to the definite one via level raising and the second reciprocity law. A similar idea plays an important role in \cite{zha14} and is used in \cite{bbv16} to obtain a result close to ours over $\Q$. More general special value formulas in the analytic rank one case for elliptic curves over $\Q$ have been established in \cite{jeskiwa17} by different methods. It is perhaps also worth pointing out that using Kolyvagin's method (in its totally real version \cite{kolo91}, \cite{nek07}) one can obtain results in the indefinite case and then deduce information in the analytic rank zero case (at least concerning the rank of the relevant Selmer group over the base field) using non-vanishing results for the first derivative of $L$-functions \cite{bufriho90}. We instead proceed in the opposite way, treating the rank zero case first and then deducing the rank one case.

\subsection{} Finally, let us mention that analogues of the reciprocity laws used in this paper have been (partially) established in other contexts \cite{liti17}, \cite{liu19}, \cite{zho19} and used to bound the ranks of the Selmer groups of suitable motives. We hope that our arguments can be of use in these settings to prove (inequalities towards) the expected special value formulas.

\subsection{Structure of the paper} In section 2 we fix our notation and introduce our main objects of interest, namely Hilbert modular forms (as well as automorphic forms on other quaternion algebras) and the associated Galois representations. In section 3 we recall the special value formulas of S. Zhang for the central value and first derivative of the $L$-functions of Hilbert modular forms. In section 4 we introduce Bloch-Kato Selmer groups for the representations of interest to us and explicitly describe the relevant local conditions in our setting. In section 5 we state our main theorem in the definite case and reduce it to a statement on finite Selmer groups. Section 6 introduces the cohomology classes and reciprocity laws needed to prove this statement. Finally, those are used to prove our main result in the definite setting in section 7, which is the heart of this paper. The indefinite case is dealt with in section 8.

\subsection{Notation and conventions}\label{nots}
We fix once for all a rational prime $p$, embeddings $\iota_\infty: \bar{\Q}\rightarrow \C, \iota_p:\bar{\Q}\rightarrow \bar{\Q}_p$ and an isomorphism $\bar{\Q}_p\xleftrightarrow{\sim}\C$ compatible with the two embeddings.\\
The symbol $Fr$ denotes \emph{geometric} Frobenius, unless stated otherwise. Accordingly, the Artin map of global class field theory is normalised so that uniformisers correspond to geometric Frobenius elements.\\
The absolute Galois group of a field $L$ is denoted by $\Gamma_L$. The completion of a number field $L$ at a place $v$ is denoted by $L_v$. If $M$ is a $\Gamma_L$-module and $c \in H^1(L, M)$ then the restriction of $c$ to $H^1(L_v, M)$ will be denoted by $loc_v(c)$.\\
We will write $M \simeq N$ to denote that two objects $M, N$ are isomorphic. The cardinality of a set $X$ will be denoted by $\# X$.\\
We let $\hat{\Z}=\varprojlim_n \Z/n\Z$ and, for an abelian group $A$, we set $\hat{A}=A\otimes_\Z\hat{\Z}$. For example the ring of finite adeles of $\Q$ is $\af=\Q\times_\Z \hat{\Z}=\hat{\Q}$.

\subsection*{Acknowledgements} The work presented in this paper was carried out during the author's PhD at the University of Duisburg-Essen, supported by SFB/TR45 “Periods, Moduli Spaces and Arithmetic of Algebraic Varieties”; he wishes to express his gratitude to all the members of the ESAGA group in Essen, and thanks the authors of \cite{blv16} for sharing a draft of their paper. This work began while the author was a guest at CIB, EPF Lausanne during the special semester “Euler systems and special values of $L$-functions”. He is very grateful to the organisers for the invitation and for the excellent working conditions offered throughout the semester. The author would also like to thank Jan Nekov\'{a}\v{r} for spotting some inaccuracies in an earlier version of this text and for providing several useful comments. This paper was completed while the author was a Research Associate at Imperial College, supported by the ERC Grant 804176.

\section{Quaternionic automorphic forms}

\subsection{} In this section we introduce the main objects we will work with, namely Hilbert modular forms and the associated Galois representations, Shimura curves and quaternionic sets. The material in this section is well known, hence we will provide no proof. For a more detailed discussion the reader is referred to \cite[Chapter 12]{nek06} and the references therein.

\subsection{Hilbert modular forms} Let us fix a totally real number field $F$ of degree $r>1$ with ring of integers $\OF$; let $G=Res_{F/\Q}GL_{2,F}$ and let $U\subset G(\af)$ be a compact open subgroup. We denote by $M(U)$ the space of Hilbert modular forms of parallel weight two and level $U$ with trivial central character and by $S(U)$ the subspace of cusp forms. They are equipped with an action of the Hecke algebra $\mathcal{H}(U\setminus G(\af)/U)$ of compactly supported, left and right $U$-invariant functions $G(\af)\rightarrow \C$.\\
Let $\n \subset \OF$ be an ideal such that $(\n, p)=(1)$. In what follows we will work with Hilbert modular forms of level $U_1(\n)$, where
\begin{equation*}
U_1(\n)=\left\{\begin{pmatrix}
a & b\\
c & d
\end{pmatrix} \in GL_2(\hat{\mathcal{O}}_F) \: : c, d-1 \equiv 0 \pmod {\hat{\n}}\right\}.
\end{equation*}

The corresponding spaces of modular (resp. cusp) forms will be denoted by $M(\n)$ (resp. $S(\n)$). Let $v$ be a finite place of $F$ not dividing $\n$ and $A_v=U_1(\n)\begin{pmatrix}
\varpi_v & 0\\
0 & 1
\end{pmatrix}U_1(\n)$, where $\varpi_v$ is a uniformiser of $F_v$. We denote by $T_v: M(\n)\rightarrow M(\n)$ the Hecke operator corresponding to the function
\begin{equation*}
\frac{1}{\int_{A_v}dg}\mathbf{1}_{A_v}
\end{equation*}
where $\mathbf{1}_{A_v}$ is the characteristic function of $A_v$ and $dg=\prod_{v \nmid \infty}dg_v$, where $dg_v$ is the Haar measure on $GL_2(F_v)$ normalised imposing $\int_{GL_2(\mathcal{O}_{F_v})}dg=1$. If $v$ divides $\n$ the same operator will be denoted by $U_v$. We denote by $\Tn\subset \mathcal{H}(U_1(\n)\setminus G(\af)/U_1(\n))$ the ring generated by the Hecke operators $T_v$ for $v \nmid \n$ and $U_v$ for $v|\n$. A cusp form $f \in S(\n)$ is called a \emph{newform} (of parallel weight two, with trivial central character) if it is an eigenvector for all the operators in $\Tn$, it is new at every place $w|\n$ and the constant term in its Fourier expansion equals one. A newform $f$ gives rise to a ring morphism
\begin{equation*}
\lambda_f:\Tn\rightarrow \C
\end{equation*}
sending an Hecke operator $T \in \Tn$ to the number $\lambda_f(T) \in \C$ such that $T\cdot f=\lambda_f(T)f$.
The $L$-function of $f$ is defined as the Euler product
\begin{equation*}
L(f, s)=\prod_{v|\n}(1-\lambda_f(U_v)N(v)^{-s})^{-1}\prod_{v \nmid \n}(1-\lambda_f(T_v)N(v)^{-s}+N(v)^{1-2s})^{-1}
\end{equation*}
yielding a holomorphic function on the half-plane $Re \: s >\frac{3}{2}$.

\subsection{Galois representations attached to newforms} Let $f \in S(\n)$ be a newform and $\mathcal{O}$ the ring generated by the eigenvalues $\lambda_f(T_v)$, $\lambda_f(U_v)$ of the Hecke operators acting on $f$. It is an order in the ring of integers of a number field $E\subset \C$ which is totally real (since $f$ has trivial central character). Thanks to the work of several people (\hspace{1sp}\cite{oht82}, \cite{wi88}, \cite{tay89}, \cite{blaro89}) one can attach to $f$ a compatible system of Galois representations

\begin{equation*}
\rho_{f, \pi}:\Gamma_F \rightarrow Aut(V_{f, \pi})
\end{equation*}
where, for each finite place $\pi$ of $E$, $V_{f, \pi}$ is a 2-dimensional vector space over the completion $E_\pi$ of $E$ at $\pi$. We denote simply by 
\begin{equation*}
\rho_{f}:\Gamma_F \rightarrow Aut(V_f)
\end{equation*}
the Galois representation corresponding to the place $\p$ of $E$ induced by the isomorphism $\C\xleftrightarrow{\sim}\bar{\Q}_p$ fixed at the beginning \ref{nots}. Let $\Op$ be the ring of integers of $E_\p$ and $\varpi$ a uniformiser of $\Op$. The representation $\rho_f$ enjoys the following properties:
\begin{enumerate}
\item it is unramified outside $\n p$;
\item for every finite place $v$ of $F$ not dividing $\n p$ we have
\begin{equation*}
det(1-Fr_v N(v)^{-s}|V_f)=1-\lambda_f(T_v)N(v)^{-s}+N(v)^{1-2s};
\end{equation*}
\item for $v \nmid \n p$ the eigenvalues of $Fr_v$ acting on $V_f$  are $v$-Weil numbers of weight $1$;
\item it is absolutely irreducible.
\end{enumerate}

By (4), the Brauer-Nesbitt theorem and the Chebotarev density theorem $\rho_f$ is uniquely characterised up to isomorphism by the property (2), which determines the trace of almost all Frobenius elements. Moreover (2) implies that
\begin{equation*}
det(V_f)=\wedge^2 V_f\simeq E_\p(-1)
\end{equation*}
hence $V_f^*=Hom(V_f, E_\p)\simeq V_f(1)$. Letting $V(f)=V_f(1)$ it follows that $V(f)$ is \emph{self-dual}, i.e. there is a skew-symmetric, non degenerate, $\Gamma_F$-equivariant pairing
\begin{equation*}
V(f)\times V(f) \rightarrow E_\p(1)
\end{equation*}
yielding an identification $V(f)\simeq Hom_{\Gamma_F}(V(f), E_\p(1))$.

We choose a $\Gamma_F$-stable $\Op$-lattice $T(f)\subset V(f)$ such that the above pairing (possibly scaled by a constant) induces a perfect pairing
\begin{equation*}
T(f)\times T(f)\rightarrow \Op(1)
\end{equation*}
hence perfect pairings
\begin{align*}
T(f)\times A(f)\rightarrow & E_\p/\Op(1)\\
T_n(f)\times A_n(f)\rightarrow & E_\p/\Op(1), \; n\geq 0,
\end{align*}
where $A(f)=V(f)/T(f)$, $A_n(f)=A(f)[\varpi^n]\simeq T_n(f)=T(f)/\varpi^n$.
\begin{ass}\label{absred}
Assume that the residual Galois representation $T_1(f)$ is irreducible (hence absolutely irreducible).
\end{ass}
Under the above assumption the isomorphism class of the Galois representations $T(f), T_n(f)$ does not depend on the choice of the lattice $T(f)$.

We will need the following information on the local structure of the $\Gamma_F$-module $V(f)$:
\begin{lem}\label{locm}(cf. \cite[12.4.4.2, 12.4.5]{nek06})
If $v$ is a place of $F$ dividing exactly $\n$ then $V(f)_{|\Gamma_{F_v}}$ is of the form
\begin{equation*}
\begin{pmatrix}
\mu \chi_{cyc} & *\\
0 & \mu
\end{pmatrix}
\end{equation*}
where $\chi_{cyc}$ is the cyclotomic character and $\mu$ is a quadratic unramified character.
\end{lem}

\subsection{Shimura curves} Let $B/F$ be a quaternion algebra split at \emph{exactly one} infinite place $\tau$ of $F$. For $U\subset \hat{B}^\times$ compact open one disposes of the space $S^{B^\times}(U)$ of (cuspidal) automorphic forms for $B^\times$ of level $U$ and weight $2$, endowed with an action of $\mathcal{H}(U\backslash \hat{B}^\times /U)$. On the other hand we can consider the Shimura curve $Y_U$ whose complex points are given by
\begin{equation}\label{compunif}
Y_U^{an}=B^\times\backslash (\C\setminus \R)\times \hat{B}^\times/U.
\end{equation}
where $B^\times$ acts on $\C\setminus \R$ via the embedding $B\subset B\otimes_{F, \tau} \R=M_2(\R)$ and the action of $GL_2(\R)$ on $\C\setminus \R$ by Möbius transformations. Every element of the Hecke algebra $\mathcal{H}(U\backslash \hat{B}^\times /U)$ gives rise to a correspondence on the curve $Y_U$, hence the Hecke algebra acts on $H^0(Y_U^{an}, \Omega_\C)$. In fact, there is a canonical, Hecke equivariant identification (see \cite[(1.6)]{nek07})
\begin{equation}\label{modshimdiff}
S^{B^\times}(U)\xrightarrow{\sim} H^0(Y_U^{an}, \Omega_\C).
\end{equation}
Together with the Hodge decomposition and the comparison between Betti and \'etale cohomology, this relates weight two cuspidal eigenforms for $B^\times$ to systems of Hecke eigenvalues in the \'etale cohomology of $Y_U$.\\
The center $\hat{F}^\times \subset \hat{B}^\times$ acts on $Y_U$. Denoting by $[z, b]$ a point of $Y_U^{an}$, where $z \in \C \setminus \R$ and $b \in \hat{B}^\times$, the action of an element $g \in \hat{F}^\times$ is given by $g \cdot [z, b]=[z, bg]$; hence the action of $\hat{F}^\times$ factors through the finite group $C_U=F^\times \backslash \hat{F}^\times/(\hat{F}^\times \cap U)$. The quotient $X_U=Y_U/C_U$ is a smooth projective scheme, and \eqref{modshimdiff} induces an identification
\begin{equation*}
S^{B^\times/Z}(U)\xrightarrow{\sim} H^0(X_U^{an}, \Omega_\C).
\end{equation*}
where we denoted by $S^{B^\times/Z}(U)\subset S^{B^\times}(U)$ the subspace of automorphic forms with trivial central character. Since those are the only automorphic forms we will consider, we will consistently work with the quotient Shimura curves $X_U$.

\subsection{Automorphic forms on totally definite quaternion algebras} Finally, we will need to work with automorphic forms (of ``parallel weight two") on totally definite quaternion algebras. Let $B/F$ be a quaternion algebra ramified at every infinite place and $U\subset \hat{B}^\times$ a compact open subgroup. Then the space of weight two automorphic forms for $B^\times$ of level $U$ is defined as
\begin{equation*}
S^{B^\times}(U)=\{f: B^\times \backslash \hat{B}^\times/U\rightarrow \C\}=\C[B^\times \backslash \hat{B}^\times/U];
\end{equation*} 
notice that $B^\times \backslash \hat{B}^\times/U$ is a finite set, which we will sometimes call a \emph{quaternionic set}. Automorphic forms with trivial central character are those which factor through $B^\times \backslash \hat{B}^\times/\hat{F}^\times U$. We will also need to work with automorphic forms modulo (powers of) $p$. For our minimal needs, it will be enough to dispose of this notion for totally definite quaternion algebras, in which case the definition is straightforward:
\begin{defin}\label{intaut}
Let $B$ be a totally definite quaternion algebra and $A$ a commutative ring. We define the space of $A$-valued automorphic forms for $B^\times$ of level $U$ as
\begin{equation*}
S^{B^\times}(U, A)=\{ f: B^\times \backslash \hat{B}^\times/U\rightarrow A \}=A[B^\times\backslash \hat{B}^\times/U]
\end{equation*}
and we define $S^{B^\times/Z}(U, A)$ by requiring $\hat{F}^\times$-invariance in addition.
\end{defin}

\section{L-functions and special value formulas}

\subsection{} The aim of this section is to recall the special value formulas due to S. Zhang relating the central value (resp. first derivative) of the $L$-function of a Hilbert newform to special points on quaternionic sets (resp. Shimura curves). These formulas are the key analytic input for our work, whose aim will be to exploit these points, and the relations among them, in order to bound the Selmer group attached to the relevant Hilbert modular form. Proofs of the formulas can be found in \cite{zh04} and, in more detail and generality, in \cite{yzz13} (see also \cite{cashti14}). 

\subsection{} Let $f \in S(\n)$ be a newform (with trivial central character); let $K/F$ be a totally imaginary quadratic extension satisfying $(\n, disc(K/F))=1$. Let $L(f_K, s)=(\Gamma_\C(s)^{[F:\Q]})L^\infty(f_K, s)$ where $L^\infty(f_K, s)$ is the $L$-function of the compatible system of Galois representations $\{V_{f, \pi}|_{\Gamma_K}\}_\pi$ and $\Gamma_\C(s)=2(2\pi)^{-s}\Gamma(s)$. The function $L(f_K, s)$ admits holomorphic continuation to the whole complex plane, and it satisfies a functional equation of the form $L(f_K, s)=\epsilon(f_K, s)L(f_K, 2-s)$. In particular the parity of the order of vanishing of $L(f_K, s)$ at $s=1$ is determined by $\epsilon(f_K, 1)$. Our assumption that $(\n, disc(K/F))=1$ implies that we can write $\n=\np\nm$, where $\np$ (resp. $\nm$) is divisible only by primes which are split (resp. inert) in $K$; let us furthermore assume that $\nm$ is squarefree. Then the value $\epsilon(f_K, 1)$, which we will simply denote by $\epsilon(f_K)$, is determined as follows:

\begin{enumerate}
\item If $r=[F: \Q]\equiv \#\{\q: \q\mid \nm\} \pmod 2$ then $\epsilon(f_K)=1$; this is called the \emph{definite case}. In this situation we will be interested in the special value $L(f_K, 1)$.
\item If $r \not \equiv \#\{\q: \q\mid \nm\} \pmod 2$ then $\epsilon(f_K)=-1$; this is called the \emph{indefinite case}. In this situation the functional equation forces the vanishing of the central value $L(f_K, 1)$, and we will look instead at $L'(f_K, 1)$.
\end{enumerate}

\subsection{The definite case.} Suppose that $[F: \Q]\equiv \#\{\q: \q\mid \nm\} \pmod 2$; let $B/F$ be the quaternion algebra ramified at all primes dividing $\nm$ as well as at all infinite places. Then $f$ can be transferred, via the Jacquet-Langlands correspondence, to an automorphic form
\begin{equation*}
f_B: B^\times\backslash \hat{B}^\times/\hat{F}^\times\hat{R}^\times\rightarrow \C
\end{equation*}
where $R\subset B$ is an Eichler order of level $\np$. We normalize $f_B$ requiring its Petersson norm to be 1 (the Petersson product being just a finite sum in this case). Fix an $R$-optimal embedding $\iota: K\hookrightarrow B$ (i.e. such that $\iota^{-1}(R)=\mathcal{O}_K$), inducing a map
\begin{equation}\label{hatiota}
\hat{\iota}: K^\times \backslash \hat{K}^\times/\hat{F}^\times \hat{\mathcal{O}}_K^\times \rightarrow B^\times\backslash \hat{B}^\times/\hat{F}^\times\hat{R}^\times;
\end{equation}
Let $a(f)=\sum_{P \in K^\times \backslash \hat{K}^\times/\hat{F}^\times \hat{\mathcal{O}}_K^\times}f_B(\hat{\iota}(P)) \in \C$.
\begin{teo}(\hspace{1sp}\cite[Theorem 7.1]{zh04})
The following equality holds:
\begin{equation}\label{specval} 
L(f_K, 1)=\frac{2^r}{\sqrt{N(disc(K/F))}}\cdot \langle f, f \rangle_{Pet}\cdot |a(f)|^2.
\end{equation}
\end{teo}

\begin{rem}\label{fixtransf}
Since we will have to work with integral automorphic forms we need to make a different choice of Jacquet-Langlands transfer $f_B$, which results in a different period appearing in the special value formula in place of $\langle f, f \rangle_{Pet}$. For a discussion of this issue we refer the reader to \cite[Section 2]{vat03} (see also \cite[Section 3.3]{lon07}). For our purposes, let us recall that we can, and will, choose $f_B \in S^{B^\times/Z}(\hat{R}^\times, \mathcal{O})$, where $\mathcal{O}$ is the ring generated by the Hecke eigenvalues of $f$, and such that the image of $f_B$ in $\Op$ contains a $\mathfrak{p}$-adic unit. This determines $f_B$ up to multiplicaiton by a $\p$-adic unit, and tor such a choice the above formula translates into
\begin{equation}\label{lalg}
C \cdot \frac{L(f_K, 1)}{\Omega^{Gr}}=|a(f)|^2
\end{equation}
where $C= \frac{\sqrt{N(disc(K/F))}}{2^r}$ and $\Omega^{Gr}=\frac{\langle f, f \rangle_{Pet}}{\eta_B}$ is the \emph{Gross period}, quotient of $\langle f, f \rangle_{Pet}$ by the \emph{congruence number} $\eta_B$.

In particular the value $C \cdot \frac{L(f_K, 1)}{\Omega^{Gr}}$ is an algebraic number, called the \emph{algebraic part} of the special value $L(f_K, 1)$ and denoted by $L^{alg}(f_K, 1)$. It follows from equation \eqref{lalg} that its $\p$-adic valuation is
\begin{equation*}
v_\p(L^{alg}(f_K, 1))=2v_\p(a(f)).
\end{equation*}
\end{rem}

\subsection{The indefinite case}\label{specvalind}
Let us now assume that $r \not \equiv \#\{\q: \q\mid \nm\} \pmod 2$; let $B/F$ be the quaternion algebra ramified at all primes dividing $\nm$ and at all but one infinite place; let $R \subset B$ be an order of conductor $\np$ and fix an $R$-optimal embedding $K \hookrightarrow B$ as before. Let $P_K \in X_{\hat{R}^\times}(\C)$ be a point with $CM$ by $\mathcal{O}_K$; via the complex uniformisation \eqref{compunif} we can take $P_K=[z, 1]$ where $z \in \C\setminus \R$ is the only point in the upper half plane fixed by the action of $K^\times \subset B^\times$. The point $P_K$ is an algebraic point of $X_{\hat{R}^\times}$, defined over the abelian extension of $K$ whose Galois group is identified with $K^\times\backslash \hat{K}^\times/\hat{F}^\times \hat{\mathcal{O}}_K^\times$ via Artin's reciprocity map.

Let $Q_K=\sum_{\sigma \in K^\times\backslash \hat{K}^\times/\hat{F}^\times \hat{\mathcal{O}}_K^\times}\sigma(P_K) \in Div(X_{\hat{R}^\times})$ and let $a(f)$ be the $f_B$-isotypical part of $Q_K - deg(Q_K)\xi \in Jac(X_{\hat{R}^\times})(K)\otimes \Q$, where $f_B \in S^{B^\times/Z}(\hat{R}^\times)$ is a Jacquet-Langlands transfer of $f$ and $\xi \in CH^1(X_{\hat{R}^\times})\otimes \Q$ is the Hodge class \cite[pag. 202]{zh04}.

\begin{teo}(\cite[Theorem 6.1]{zh04})
The following equality holds:
\begin{equation*}
L'(f_K, 1)=\frac{2^{r+1}}{\sqrt{N(disc(K/F))}}\cdot \langle f, f \rangle_{Pet}\cdot \langle a(f), a(f)\rangle_{NT}
\end{equation*}
where $\langle -, -\rangle_{NT}$ is the Neron-Tate height.
\end{teo}

\begin{rem}
Let $\mathbb{T}^{B^\times/Z}_{\np} =\Op[T_v, v \nmid \n, U_v, v \mid \n]$ where $T_v$ (resp. $U_v$) is the characteristic function of the double coset $[\hat{R}^\times \varpi_v \hat{R}^\times]$, with $\pi_v \in B_v^\times\subset \hat{B}^\times$ an element of norm $N(v)$. Under assumption \ref{absred} the maximal ideal of $\mathbb{T}^{B^\times}_{\np}$ containing the kernel $I_{f_B}$ of the map $\mathbb{T}_{\np}^{B^\times/Z}\rightarrow \Op$ attached to $f_B$ is not Eisenstein, so the Hodge class dies modulo $I_{f_B}$. Hence inside $(CH^1(X_{\hat{R}^\times})(K))\otimes \Op)/I_{f_B}=(Jac(X_{\hat{R}^\times})(K)\otimes \Op)/I_{f_B}$ we have $a(f)=[Q_K]$.
\end{rem}

\section{Selmer groups}\label{selgps}
\subsection{}\label{fixf} Throughout this section we fix a Hilbert newform $f \in S(\n)$ and a $CM$ extension $K/F$ such that $\n$ is squarefree and all its prime factors are inert in $K$ (in the notation of the previous section, we are assuming that $\n=\nm$).
\subsection{Bloch-Kato Selmer groups} Let $V=V(f)$. Recall that this is a two-dimensional $E_\p$-vector space with a continuous $\Gamma_F$-action, inside which we have chosen self-dual $\Op$-lattice $T(f)$. We defined $A(f)=V(f)/T(f)$ and, for $n \geq 1$, we denote $A_n(f)=A(f)[\varpi^n]\simeq T_n(f)=T(f)/\varpi^n$.

The Bloch-Kato Selmer group of the $\Gamma_K$-module $V$ is defined as
\begin{equation*}
Sel(K, V)=ker\left(H^1(K, V)\rightarrow \prod_v\frac{H^1(K_v, V)}{H^1_f(K_v, V)} \right)
\end{equation*}
where, for a finite place $v$ of $K$,
\begin{equation*}
H^1_f(K_v, V)=ker\left(H^1(K_v, V)\rightarrow\begin{cases}
H^1(I_v, V) \text{ if } v \nmid p\\
H^1(K_v, V\otimes B_{cris}) \text{ if } v \mid p
\end{cases}\right).
\end{equation*}

We also define Selmer groups $Sel(K, M)$ for $M=T(f), T_n(f), A(f), A_n(f)$ imposing as local conditions $H^1_f(K_v, M)$ those coming from $H^1_f(K_v, V)$ by propagation. In particular under the local Tate pairing at a place $v$
\begin{equation*}
H^1(K_v, T_n(f))\times H^1(K_v, A_n(f))\rightarrow E_\p/\Op
\end{equation*}
the local conditions $H^1_f(K_v, T_n(f))$ and $H^1_f(K_v, A_n(f))$ are annihilators of each other, since the same is true for the Bloch-Kato local conditions on $V$.

\subsection{} Our aim is to determine the local conditions defining $Sel(K, A_n(f))$ more explicitly. Precisely, imposing suitable hypotheses on the Galois representation $T_1(f)$, we wish to describe these local conditions \emph{purely in terms of the Galois representation} $A_n(f)$ and of the level $\n$.

\subsection{Local condition at places outside $\n p$} If $v$ is a finite place of $K$ not dividing $\n p$ then $V(f)$ is unramified at $v$. Since the unramified local condition is stable under propagation on unramified $\Gamma_K$-modules it follows that the local condition on $H^1(K_v, M)$ for $M=A_n(f), T_n(f)$ is the unramified one, i.e.
\begin{equation*}
H^1_f(K_v, M)=ker\left((H^1(K_v, M)\rightarrow
H^1(I_v, M)\right).
\end{equation*}

\subsection{Local condition at places dividing $\n$} We want to express local conditions at places dividing $\n$ in terms of the Galois module $A_n(f)$. To do this we will need the following
\begin{ass}\label{nmass}
Assume that, if $\q|\n$ and $N(\q) \equiv \pm 1 \pmod p$, then $T_1(f)$ is ramified at $\q$.
\end{ass}

\begin{lem}(cf. \cite[Lemma 3.5]{pw11})\label{locnm}
Let $\q|\n$ and $n\geq 1$. Under the above assumption there exists a unique submodule $A_n^{\q}(f)$ of $A_n(f)$ free of rank one over $\Op/\varpi^n$ on which $\Gamma_{K_\q}$ acts via the cyclotomic character.
\end{lem}
\begin{proof}
For $n=1$ this follows from assumption \ref{nmass}. Indeed by Lemma \ref{locm} the $\Gamma_{K_\q}$-module $A_1(f)$ is of the form
\begin{equation*}
\begin{pmatrix}
\chi_{cyc} & *\\
0 & 1
\end{pmatrix}
\end{equation*}
If $A_1(f)^{I_\q}$ is one dimensional, which is always the case when $N(\q) \equiv \pm 1 \pmod p$, then it is the only subspace on which $\Gamma_{K_\q}$ acts via the cyclotomic character. Otherwise $N(\q)\not \equiv \pm 1 \pmod p$, so $Fr_\q$ acting on $A_1(f)^{I_\q}$ has distinct eigenvalues, hence the claim follows. The statement for general $n$ can then be established by induction.
\end{proof}

\begin{prop}\label{nmprop}
With the notations of the previous lemma, we have:
\begin{equation*}
H^1_f(K_\q, A_n(f))=Im(H^1(K_\q, A_n^\q(f))\rightarrow H^1(K_\q, A_n(f))).
\end{equation*}
\end{prop}
\begin{proof}
First of all, we have $H^1(K_\q, V(f))=0$, $H^1_f(K_\q, A(f))=0$ and $H^1(K_\q, T(f))=H^1_f(K_\q, T(f))$ is finite (cf. \cite[Proposition 2.7.8]{nek12}), hence $H^1_f(K_\q, T_n(f))=Im(H^1(K_\q, T(f))\rightarrow H^1(K_\q, T_n(f)))$. Set $T=T(f)_{|\Gamma_{K_\q}}$. Then we have an exact sequence of $\Op[\Gamma_{K_\q}]-$modules
\begin{equation*}
0\rightarrow T^+\rightarrow T \rightarrow T^-\rightarrow 0
\end{equation*}
where $T^+=\Op(1)$ and $T^-=\Op$. The induced long exact sequence in cohomology yields
\begin{equation*}
0\rightarrow H^0(K_\q, T^-)\rightarrow H^1(K_\q, T^+)\rightarrow H^1(K_\q, T) \rightarrow H^1(K_\q, T^-)\rightarrow H^2(K_\q, T^+);
\end{equation*}
moreover we have $H^0(K_\q, T^-)\simeq H^2(K_\q, T^+)\simeq \Op$, $H^1(K_\q, T^+)=K_\q^\times \hat\otimes \Op$ and $H^1(K_\q, T^-)=Hom_{cont}(\Gamma_{K_\q}^{ab}, \Op)=Hom_{cont}(\widehat{{K_\q}^\times}, \Op)$. As $H^1(K_\q, T)$ is finite and $H^1(K_\q, T^-)$ is infinite the second coboundary map is non zero, hence injective. It follows that the map $T^+\rightarrow T$ induces a surjection $H^1(K_\q, T^+)\twoheadrightarrow H^1(K_\q, T)$; furthermore the map $H^1(K_\q, T^+)=K_\q^\times \hat\otimes \Op\rightarrow H^1(K_\q, T^+/\mathfrak{p}^n)=K_\q^\times \otimes \Op/\mathfrak{p}^n$ is surjective. Hence we obtain $H^1_f(K_\q, T_n(f))=Im(H^1(K_\q, T^+/\mathfrak{p}^n)\rightarrow H^1(K_\q, T_n(f)))$, from which the proposition follows.
\end{proof}

\subsection{Local conditions at places above $p$} Fix a place $v$ of $K$ lying above $p$. Then the Bloch-Kato local condition at $v$ can be described in terms of flat cohomology of $p$-divisible groups. This is discussed in detail in appendix A of \cite{nek12}; the key facts that we need are summarised in the following

\begin{prop}\label{condlocpflat}
\begin{enumerate}
\item Let $v$ be a place of $F$ above $p$. Then there exists a $p$-divisible group $\mathcal{G}/\mathcal{O}_{F_v}$ with an action of $\Op$ such that $T_p(\mathcal{G})=T(f)_{|\Gamma_{F_v}}$.
\item For $\mathcal{G}$ as in $(1)$ and $n \geq 1$ we have
\begin{equation*}
H^1_f(K_w, T_n(f))=H^1_{fl}(\mathcal{O}_{K_w}, \mathcal{G}[\varpi^n])
\end{equation*}
where $w$ is a place of $K$ above $v$.
\item Assume that $p$ is unramified in $K$. Let $w$ be a place of $K$ above $p$ and $\mathcal{H}/\mathcal{O}_{K_w}$ a finite flat group scheme with $\Op/(\varpi^n)$-action such that $T_n(f)\simeq \mathcal{H}(\bar{K}_w)$ as $\Op/\varpi^n[\Gamma_{K_w}]$-modules. Then $H^1_f(K_w, T_n(f))=H^1_{fl}(\mathcal{O}_w, \mathcal{H})$.
\end{enumerate}
\end{prop}

\begin{proof}
The existence of $\mathcal{G}$ such that $T_p(\mathcal{G})=T(f)_{|G_{F_v}}$ is proved in \cite[Theorem 1.6]{tay95}.

The second point follows from \cite[A.2.6]{nek12}, which is a consequence of the fact that $T_p(\mathcal{G})\otimes \Q_p$ is a crystalline $\Gamma_{F_v}$-representation, of local flat duality and of the vanishing of the second cohomology groups $H^2_{fl}(\mathcal{O}_{K_w}, \mathcal{G}[\varpi^k]), k \geq 1$. 

Finally let us prove the third statement. We have isomorphisms $\mathcal{G}[\varpi^n](\bar{K}_w) \simeq T_n(f)\simeq \mathcal{H}(\bar{K}_w)$. Since $p$ is odd and unramified in $K$, by \cite[Corollary 3.3.6]{ra74} the isomorphism $\mathcal{G}[\varpi^n](\bar{K}_w)\simeq \mathcal{H}(\bar{K}_w)$ comes from an isomorphism $\mathcal{G}[\varpi^n]\simeq \mathcal{H}$, under which $H^1_{fl}(\mathcal{O}_w, \mathcal{G}[\varpi^n])$ and $H^1_{fl}(\mathcal{O}_w, \mathcal{H})$ are identified in $H^1(K_w, T_n(f))$.
\end{proof}

\subsection{} The outcome of this section is that we have described the local conditions giving $Sel(K, T_n(f))\simeq Sel(K, A_n(f))$ purely in terms of the $\Gamma_K$-module $A_n(f)$ and of the level $\n$ of $f$; this will be important for us as in what follows we will realise $T_n(f)$ as a quotient of the Tate module of the Jacobian of several Shimura curves.

\section{Statement of the main theorem, and a first d\'evissage}

\subsection{} Fix $f \in S(\n)$ as in \ref{fixf}. Until further notice, we are now going to work in the \emph{definite case}, i.e. we suppose that $[F: \Q]\equiv \# \{\q, \q \mid \n\} \pmod 2$. Then the sign of the functional equation of $L(f_K, 1)$ is 1, and Zhang's special value formula \eqref{specval} implies that, with the notations as in \ref{fixtransf}, $v_\p(L^{alg}(f_K, 1))=2v_\p(a(f))$. In this setting, our aim is to prove the following result

\begin{teo}\label{thbkdef}
Let $f \in S(\n)$. Assume that
\begin{enumerate}
\item The level $\n$ of $f$, the discriminant $disc(K/F)$ and the prime $p$ below $\p$ are coprime to each other. Moreover $p>3$ is unramified in $F$, and $\n$ is squarefree and all its factors are inert in $K$.
\item The image of the residual Galois representation $\bar{\rho}: \Gamma_F\rightarrow Aut(T_1(f))$ attached to $f$ contains $SL_2(\mathbf{F}_p)$.
\item For every prime $\q\mid \n$ we have $N(\q) \not \equiv -1 \pmod p$. Moreover if $N(\q) \equiv 1 \pmod p$ then $\bar{\rho}$ is ramified at $\q$.
\item $L(f_K, 1)\neq 0$.
\end{enumerate}
Then $Sel(K, A(f))$ is finite and the following inequality holds:
\begin{equation*}
length_{\Op} Sel(K, A(f))\leq v_\p(L^{alg}(f_K, 1)).
\end{equation*}
Moreover the inequality is an equality if the implication stated at the end of Theorem \ref{mainthm} holds true.
\end{teo}

\begin{notat}
For $a \in \Op\setminus \{0\}$ we denote by $ord_\varpi(a)$ its $\varpi$-adic valuation. More generally, if $M$ is an $\Op$-module of finite type and $m \in M\setminus \{0\}$ we let $ord_{\varpi}(m)=sup\{n\geq 0: m \in \varpi^nM\}$. The length of an $\Op$-module $M$ will be denoted by $\lOp(M)$. Hence $ord_{\varpi}(a)=\lOp(\Op/(a))$ for $a \in \Op\setminus \{0\}$, and the cardinality of a finite $\Op$-module $M$ equals $(\#\Op/(\varpi))^{\lOp(M)}$.
\end{notat}

\subsection{} With the above notation, our aim is to prove the inequality

\begin{equation*}
\lOp Sel(K, A(f))\leq ord_{\varpi}(L^{alg}(f_K, 1))
\end{equation*}

under the assumption that the right hand side is not equal to infinity. In other words we have to prove that

\begin{equation}\label{maineq}
\lOp Sel(K, A(f))\leq 2ord_{\varpi}(a(f)).
\end{equation}

\begin{rem}\label{bkcomp}
Bloch and Kato \cite{bk90} predict that in our situation the following formula holds:
\begin{equation*}
ord_\varpi\left(\frac{L(f_K, 1)}{\Omega^{BK}}\right)=\lOp Sel(K, A(f))+ ord_\varpi\left(\prod_\q t_\q\right)
\end{equation*}
where $t_\q$ is the $\q$-th Tamagawa number of $A(f)$ and $\Omega^{BK}$ is a suitable period. In our sought-for formula \eqref{maineq} Tamagawa number are missing. The point is that the period $\Omega^{Gr}$ in Zhang's special value formula is different from the one showing up in the Bloch-Kato conjecture. To show that our formula \eqref{maineq} -better, the corresponding equality - is equivalent to the one predicted by Bloch and Kato one needs to compare the quantities $\Omega^{Gr}$ and $\Omega^{BK}$. This is done in \cite[Theorem 6.8]{pw11} for modular forms over $\Q$; there the ratio between the two periods is shown to be equal precisely to the product of the missing Tamagawa numbers. We do not know whether the analogous result over totally real fields has been proved, and we did not address this issue.
\end{rem}

Let us show first of all that it is enough to prove a mod-$\varpi^n$ version of the inequality \eqref{maineq}.
\begin{lem}\label{enmodp}
Assume that $L(f_K, 1)\neq 0$ and that the inequality
\begin{equation}\label{modneq}
\lOp Sel(K, A_n(f)) \leq 2 ord_\varpi(a_n(f))
\end{equation}
holds true for infinitely many $n$, where $a_n(f)\in \Op/\varpi^n$ denotes the reduction of $a(f)$. Then
\begin{equation*}
\lOp Sel(K, A(f)) \leq 2ord_{\varpi}(a(f)).
\end{equation*}
Moreover equality in the last equation holds if it does in \eqref{modneq} for infinitely many $n$.
\end{lem}
\begin{proof}
If $L(f_K, 1)\neq 0$ then $a(f)\neq 0$, hence $a(f)\not \equiv 0 \pmod {\varpi^n}$ for $n$ large enough. For any such $n$ we have $ord_\varpi a(f)=ord_\varpi a_n(f)$. By hypothesis we have the inequality $\lOp Sel(K, A_n(f))\leq 2ord_\varpi a_n(f)$ for infinitely many $n$. Now $A_n(f)=A(f)[\varpi^n]$, and by the next control result (Proposition \ref{cnt}) we have the equality 
\begin{equation*}
Sel(K, A(f)[\varpi^n])=Sel(K, A(f))[\varpi^n].
\end{equation*}
Hence we obtain, for infinitely many $n$:
\begin{equation*}
\lOp Sel(K, A(f))[\varpi^n]=\lOp Sel(K, A_n(f))\leq 2ord_\varpi (a_n(f))=2ord_\varpi (a(f))\\
\end{equation*}
As $ord_\varpi (a(f)) < \infty$ it follows that $\lOp Sel(K, A(f))[\varpi^n]$ is constant for $n$ large, hence for such $n$ we have $Sel(K, A(f))=Sel(K, A_n(f))$ and the lemma follows.
\end{proof}

\begin{prop}(cf. \cite[Lemma 3.5.3]{mr04})\label{cnt}
For $n \geq 1$ the natural map
\begin{equation*}
Sel(K, A(f)[\varpi^n])\longrightarrow Sel(K,A(f))[\varpi^n]
\end{equation*}
is an isomorphism.
\end{prop}
\begin{proof}
To shorten the notation let us denote $A(f)$ by $M$ in this proof. Let $\Sigma$ be the set consisting of all infinite places of $K$ and all places dividing $\n p$. Then we have the following commutative diagram with exact rows:
\begin{center}
\begin{tikzcd}
Sel(K, M[\varpi^n]) \arrow[r, hook] \arrow[d] & H^1(K_\Sigma/K, M[\varpi^n]) \arrow[r] \arrow[d] & \oplus_{v \in \Sigma} H^1(K_v, M[\varpi^n])/H^1_f(K_v, M[\varpi^n])\arrow[d]\\
Sel(K,M)[\varpi^n] \arrow[r, hook] & H^1(K_\Sigma/K, M)[\varpi^n] \arrow[r] & \oplus_{v \in \Sigma} H^1(K_v, M)/H^1_f(K_v, M)
\end{tikzcd}
\end{center}
where $K_\Sigma/K$ is the maximal extension unramified outside $\Sigma$.

Since the Selmer structure on $M[\varpi^n]$ is propagated from the Selmer structure on $M$, the rightmost vertical map in injective. Therefore by the snake lemma it is enough to show that the central vertical map is an isomorphism. We have an exact sequence:
\begin{equation*}
0 \longrightarrow M[\varpi^n]\longrightarrow M\xrightarrow{\cdot \varpi^n} M \longrightarrow 0.
\end{equation*}
Taking the long exact sequence in cohomology we find an exact sequence:
\begin{equation*}
H^0(K_\Sigma/K, M) \longrightarrow H^1(K_\Sigma/K, M[\varpi^n]) \longrightarrow H^1(K_\Sigma/K, M)\xrightarrow{\varpi^n} H^1(K_\Sigma/K, M).
\end{equation*}
To end the proof it suffices to notice that $H^0(K, M)=0$ since $M[\varpi]=A_1(f)$ is irreducible.
\end{proof}

\subsection{} We will prove the inequality in Lemma \ref{enmodp}, hence Theorem \ref{thbkdef}, exploiting a system of cohomology classes belonging to $H^1(K, T_n(f))$; their construction relies on level raising of $f$ modulo $\varpi^n$ at suitable \emph{admissible primes}, introduced in the next section, allowing to realise $T_n(f)$ in the cohomology of several Shimura curves.

\section{Explicit reciprocity laws}\label{reclaws}

\subsection{} Fix $f \in S(\n)$ and $K/F$ as in \ref{fixf}.

\begin{defin}\label{defadm}
Let $n\geq 1$. A prime $\pl$ of $\mathcal{O}_F$ is called $n$-admissible if:
\begin{enumerate}
\item $\pl \nmid p \: \n $.
\item $\pl$ is inert in $K$.
\item $p \nmid N(\pl)^2-1$.
\item $(N(\pl)+1)^2\equiv \lambda_{f}(T_\pl)^2\pmod {\varpi^n}$.
\end{enumerate}
\end{defin}

\begin{notat}
In what follows we will work with certain automorphic forms modulo $\varpi^n$ on totally definite quaternion algebras. We will deal as always only with automorphic forms with trivial central character; furthermore the level of the automorphic forms we will consider will always come from a maximal order. To shorten our notation, if $B/F$ is a totally definite quaternion algebra, we will denote by $S^{B^\times}(\Op/\varpi^n)$ the space of automorphic forms on $B^\times$ with trivial central character and of level $\hat{R}^\times$ where $R\subset B$ is a maximal order. In other words, $S^{B^\times}(\Op/\varpi^n)=\{f: B^\times \backslash \hat{B}^\times/\hat{R}^\times\hat{F}^\times \rightarrow \Op/\varpi^n\}$.
\end{notat}

\begin{defin}\label{defadmeig}
\begin{enumerate}
\item An eigenform $g \in S^{B^\times}(\Op/\varpi^n)$ is called $n$-admissible if $B$ is a totally definite quaternion algebra of discriminant $\mathfrak{D}_g$ divisible by $\n$ and by $n$-admissible primes, $g$ is non zero modulo $\varpi$ and the Hecke eigenvalues of $g$ for the Hecke operators outside $\mathfrak{D}_g/\n$ are equal to those of $f$ modulo $\varpi^n$ . If $k \leq n$, the reduction modulo $\varpi^k$ of $g$ will be denoted by $g_k$.
\item For an admissible eigenform $g \in S^{B^\times}(\Op/\varpi^n)$ we define $a(g)=\sum_{P \in K^\times \backslash \hat{K}^\times/\hat{F}^\times \hat{\mathcal{O}}_K^\times}g(\hat{\iota}(P)) \in \Op/\varpi^n$, where $\hat{\iota}$ is defined as in \eqref{hatiota}.
\end{enumerate}
\end{defin}

\begin{lem}\label{structadm}(cf. \cite[p. 328]{lon12})
Let $\pl$ be an $n$-admissible prime. Then:
\begin{enumerate}
\item $T_n(f)\simeq \Op/\varpi^n\oplus  \Op/\varpi^n(1)$ as $\Gamma_{K_\pl}$-modules, and this decomposition is unique.
\item $H^1(K_\pl, T_n(f))\simeq H^1(K_\pl, \Op/\varpi^n)\oplus H^1(K_\pl, \Op/\varpi^n(1))$ where both direct summands are free $\Op/\varpi^n$-modules of rank one, and the first one is identified with the unramified cohomology group $H^1_{ur}(K_\pl, T_n(f))$.
\end{enumerate}
\end{lem}

\begin{notat}
We denote the summand $H^1(K_\pl, \Op/\varpi^n(1))$ in the above decomposition by $H^1_{tr}(K_{\pl}, T_n(f))$, so that $H^1(K_\pl, T_n(f))=H^1_{ur}(K_\pl, T_n(f))\oplus H^1_{tr}(K_\pl, T_n(f))$.
\end{notat}

\begin{proof} The direct sum decomposition in $(1)$ comes from the fact that, by $(2)$ and $(4)$ in Definition \ref{defadm}, the polynomial $det(1-xFr_{K, \pl}|T_n(f))$ splits as a product $(1-x)(1-N_{F/\Q}(\pl)^{-2}x)$. Moreover $(3)$ guarantees that $1\not \equiv N_{F/\Q}(\pl)^2\pmod {\varpi^n}$, hence the decomposition is unique. This also implies that $H^1_{ur}(K_\pl, T_n(f))=T_n(f)/(Fr_{K, \pl}-1)T_n(f)=\Op/\varpi^n$. Finally, the quotient $H^1(K_\pl, T_n(f))/H^1_{ur}(K_\pl, T_n(f))$ equals $Hom(I_\pl, T_n(f))^{\Gamma_{K_\pl}}$. Any such morphism factors through the tame inertia, and has image contained in $\Op/\varpi^n(1)$ since $Fr_{K, \pl}$ acts on the tame inertia as multiplication by $N_{F/\Q}(\pl)^{-2}$.
\end{proof}

\begin{notat}\label{changesel}
Let $n \geq 1$, $M=T_n(f)\simeq A_n(f)$ and let $\mathfrak{a}$ be a product of distinct $n$-admissible primes. We denote by $Sel_{\mathfrak{a}}(K, M)$ (resp. $Sel^{\mathfrak{a}}(K, M)$, $Sel_{(\mathfrak{a})}(K, M)$) the Selmer group defined by the same local conditions as those for $M$ at all places except at $\pl \mid \mathfrak{a}$, where the local condition is the zero subspace (resp. $H^1(K_\pl, M)$, $H^1_{tr}(K_\pl, M)$). If $\mathfrak{b}$ is a product of distinct $n$-admissible primes not dividing $\mathfrak{a}$, we will combine the above notations with the obvious meaning. For example, we denote by $Sel^\mathfrak{b}_{(\mathfrak{a})}(K, M)$ the Selmer group obtained imposing as local condition at primes dividing $\mathfrak{a}$ (resp. $\mathfrak{b}$) the transverse (resp. relaxed) one.\\
Let us point out that, if $g \in S^{B^\times}(\Op/\varpi^n)$ is an $n$-admissible form which is the reduction of the Jacquet-Langlands transfer of a Hilbert newform $\tilde{f} \in S(\mathfrak{D}_g)$, then by the discussion in section \ref{selgps} we have $Sel(K, T_n(\tilde{f}))=Sel_{(\mathfrak{D}_g/\mathfrak{n})}(K, T_n(f))$.
\end{notat}

\subsection{} As remarked in the previous section, in order to prove Theorem \ref{thbkdef} it suffices to prove the inequality \eqref{modneq} for arbitrarily large $n$. In light of this, Theorem \ref{thbkdef} will follow from the next result, taking $g$ to be the reduction  modulo $\varpi^n$ of a (suitably normalised) Jacquet-Langlands transfer of $f$ to the totally definite quaternion algebra with discriminant $\n$. The proof of the theorem is the object of the next section.

\begin{teo}\label{mainthm}
Let $f \in S(\n)$ and let $K/F$ be a $CM$ extension satisfying assumptions $(1)$, $(2)$, $(3)$ of Theorem \ref{thbkdef}.
Let $n=2k$ and let $g \in S^{B^\times}(\Op/\varpi^n)$ be an admissible eigenform such that $a(g)\not \equiv 0 \pmod {\varpi^k}$. Then the following inequality holds:
\begin{equation*}
\lOp Sel_{(\mathfrak{D}_g/\n)}(K, A_k(f))\leq 2ord_{\varpi}(a(g_k)).
\end{equation*}
Moreover the above inequality is an equality provided that the following implication holds true: if $h$ is an admissible automorphic form mod $\varpi$ and $Sel_{(\mathfrak{D}_h/\n)}(K, A_1(f))=0$ then $a(h) \in \Op/\varpi$ is non zero.
\end{teo}

\begin{rem}\label{reverseimp}
\begin{enumerate}
\item As it will become clear later (see Remark \ref{freeness}), we have to work modulo $\varpi^{2k}$ in order to establish a result modulo $\varpi^k$ because we will need to make use of a certain \emph{freeness} property of the Euler system we construct. This is immaterial as long as we are interested in the special value formula \eqref{maineq}, which concerns modular forms in characteristic zero.
\item Let us say a word on the condition allowing to promote the inequalities in the above theorem to equalities. Lift $h$ to an eigenform in characteristic zero \cite[Lemma 6.11]{dese74}, and let $\tilde{f}$ be the Hilbert modular form obtained as Jacquet-Langlands transfer of a lift. In light of Zhang's special value formula and the observation in \ref{changesel}, what we need to know is the implication
\begin{equation*}
Sel(K, A(\tilde{f}))=0\Longrightarrow L^{alg}(\tilde{f}_K, 1) \text{ is a unit.}
\end{equation*}
This is currently deduced from Skinner-Urban's divisibility in the relevant Iwasawa main conjecture \cite{su14} (proved in the ordinary case for Hilbert modular forms by Wan \cite{wan15}).

However we would like to point out that the result in \emph{loc. cit.} is stronger than what we need, and our theorem shows that a generalisation to $GL_{2,F}$ of Ribet's converse of Herbrand's theorem \cite{rib76} would suffice to obtain the sought-for equality.
\end{enumerate}
\end{rem}

\subsection{} We will now work with the notations and assumptions of Theorem \ref{mainthm}: in particular we are given an admissible eigenform $g \in S^{B^\times}(\Op/\varpi^n)$ on a definite quaternion algebra $B$ of discriminant $\mathfrak{D}_g$. It determines an $\Op/\varpi^n$-valued character of the Hecke algebra whose kernel will be denoted by $I_g$, and a surjective map $\psi_g: S^{B^\times}(\Op)\rightarrow \Op/\varpi^n$ sending $h$ to $\langle g, h \rangle_{\hat{R}^\times}$ where the pairing $\langle \cdot, \cdot \rangle_{\hat{R}^\times}$ is defined in \cite[3.5]{wa15}.\\
The next Theorem \ref{raiselev} collects the essential ingredients needed to construct the cohomology classes which we will use in our Euler system argument. It is a level raising result at an admissible prime $\pl\nmid \mathfrak{D}_g$, stating that the representation $T_n(f)$ appears in the mod $\varpi^n$-cohomology of the (quotient) Shimura curve $X_\pl$ with full level structure attached to a quaternion algebra $B_\pl$ of discriminant $\mathfrak{D}_g \pl$. Before stating the result we need to introduce some notation. Let $\mathbb{T}^{B_\pl^\times} =\Op[T_v, v \nmid \mathfrak{D}_g\pl, U_v, v \mid \mathfrak{D}_g\pl]$. Define $\lambda_\pl: \mathbb{T}^{B_\pl^\times} \rightarrow \Op/\varpi^n$ by sending operators different from $U_\pl$ to the corresponding eigenvalue for their action on $g$, and sending $U_\pl$ to the value $\epsilon_\pl \in \{\pm 1\}$ such that $N(\pl)+1 \equiv \epsilon_\pl \lambda_f(T_\pl) \pmod {\varpi^n}$. Let $I_\pl=Ker(\lambda_\pl)$. Let $J_\pl$ be the Jacobian of $X_\pl$ and $\phi_\pl$ the group of connected components of the special fibre at $\pl$ of its Néron model.

\begin{teo}\label{raiselev}
There is an isomorphism of $\Gamma_F$-modules
\begin{equation}\label{isorepn}
(T_p(J_\pl)\otimes_{\Z_p}\Op)/I_\pl \simeq T_n(f).
\end{equation}
Furthermore there are isomorphisms $(\phi_\pl\otimes\Op)/I_\pl \simeq S^{B^\times}(\Op)/I_g\simeq \Op/\varpi^n$, the last one being induced by $\psi_g$, and a commutative diagram
\begin{center}
\begin{tikzcd}
J_\pl(K_\pl)/I_\pl \arrow[r, "\kappa"] \arrow[d] & H^1(K_\pl, T_n(f))\arrow[d]\\
(\phi_\pl\otimes\Op)/I_\pl \arrow[r, "\simeq"] & H^1_{tr}(K_{\pl}, T_n(f))
\end{tikzcd}
\end{center}
where the map $\kappa$ is induced by the Abel-Jacobi map and the isomorphism \eqref{isorepn}.
\end{teo}

\begin{rem} We will not enter into the details of the proof of the above theorem, which already appeared few times in the literature. It was first proved over totally real fields by Longo \cite{lon07}, following the strategy in \cite[Section 5]{bd05}, under the assumption that $f$ is $p$-isolated \cite[Definition 3.2]{lon07}. However, as remarked in \cite{chi15} (generalised to totally real fields in \cite[Theorem 5.3, 5.4, 5.7]{wa15}), one only needs to know that $S^{B^\times}(\Op)/I_{g}\simeq \Op/\varpi^n$. This follows from \cite[Theorem 1.1]{ma19}, whose hypotheses are satisfied in our case: indeed $\bar{\rho}$ is Steinberg at primes dividing $\n$ by Lemma \ref{locm}, and at primes dividing $\mathfrak{D}_g/\n$ by definition of admissible prime. Furthermore, the Taylor-Wiles condition \cite[4. Theorem 1.1]{ma19} is implied by our large image assumption (2) in Theorem \ref{thbkdef}. Finally assumption (3) guarantees that the multiplicity $k$ in \cite[Theorem 1.1]{ma19} equals zero.
\end{rem}

\subsection{Construction of the cohomology class $c(\pl)$.} The notations being as in Theorem \ref{raiselev}, let $Q_K=\sum_{\sigma \in K^\times\backslash \hat{K}^\times/\hat{F}^\times \hat{\mathcal{O}}_K^\times}\sigma(P_K) \in (CH^1(X_\pl)(K)\otimes \Op)/I_{\pl}\simeq (Pic (X_\pl)(K)\otimes \Op)/I_{\pl}$, where $P_K \in X_\pl(\C)$ is a point with $CM$ by $\mathcal{O}_K$. The point $Q_K$ gives rise to a cohomology class $c(\pl) \in H^1(K, (T_p(J_\pl)\otimes_{\Z_p}\Op)/I_\pl)=H^1(K, T_n(f))$.

\subsection{Localisation of $c(\pl)$ at $\pl$: the first reciprocity law.} The first key observation underlying the method introduced in \cite{bd05} is that the cohomology class $c(\pl)$ belongs to $Sel_{(\mathfrak{D}_g\pl/\n)}(K, T_n(f))$, i.e. its localisation at primes dividing $\mathfrak{D}_g/\n$ \emph{and at the additional prime} $\pl$ falls in the \emph{transverse} part. Furthermore the failure for the localisation of $c(\pl)$ at $\pl$ being zero is measured by $a(g)$:

\begin{teo}(First reciprocity law)
\begin{enumerate}
\item $c(\pl) \in Sel_{(\mathfrak{D}_g\pl/\n)}(K, T_n(f))$.
\item $loc_\pl (c(\pl)) \in H^1_{tr}(K_\pl, T_n(f))\simeq \Op/\varpi^n$ and we have an equality, up to $\varpi$-adic unit:
\begin{equation*}
loc_\pl(c(\pl))=a(g).
\end{equation*}
\end{enumerate}
\end{teo}
\begin{proof}
If $v$ is a place of $K$ not dividing $\mathfrak{D}_g\pl p$ then the Shimura curve $X_\pl$ has good reduction at $v$, hence $loc_v(c(\pl)) \in H^1_{ur}(K_v, T_n(f))$. If $v \mid \n$ then the Jacobian of $X_\pl$ has purely toric reduction at $v$; with the notation as in the proof of Proposition \ref{nmprop}, it follows that $loc_v(c(\pl)) \in Im(H^1(K_\q, T^+/\mathfrak{p}^n)\rightarrow H^1(K_\q, T_n(f)))=H^1_f(K_\q, T_n(f))$ (cf. \cite[(4.17)]{wa15}, \cite[Lemma 8]{gp12}), where the equality follows from Proposition \ref{nmprop}. For the same reason we have that $loc_\mathfrak{q}(c(\pl))\in H^1_{tr}(K_\q, T_n(f))$ for every $\q \mid \mathfrak{D}_g\pl/\n$.\\
For a place $v$ above $p$, since the Jacobian $J_{\pl}$ has good reduction at $v$, the image of the Kummer map in $H^1(K_v, J_\pl[p^n])$ lies in $H^1_{fl}(\mathcal{O}_{K_v}, \mathcal{J}_\pl[p^n])$, where $\mathcal{J}_\pl$ is the Neron model of $J_\pl$. This can be proved by a direct generalisation of \cite[Proposition 3.2]{lv10} (see also \cite[Lemma 7]{gp12}). Since $K$ is unramified at $v$ the map $J_\pl[p^n]\rightarrow T_n(f)$ induced by $T_p(J_\pl)\otimes_{\Z_p}\Op/I_\pl\simeq T_n(f)$ comes from a map $\mathcal{J}_\pl[p^n]\rightarrow \mathcal{G}$ where $\mathcal{G}$ is a finite flat group scheme with generic fiber $T_n(f)$. The description of the finite condition at places above $p$ in Proposition \ref{condlocpflat} then shows that $loc_v(c)\in H^1_f(K_v, T_n(f))$.\\
Finally, the equality (up to unit) in $(2)$ is proved in \cite[Proposition 3.9]{lon07} (see also \cite[Theorem 6.2]{wa15} for the Iwasawa-theoretic version and \cite[Proposition 2.8.3]{nek12} for a more general statement); besides Theorem \ref{raiselev}, it rests on the study of the bad reduction of the Shimura curve $X_{\pl}$ at the prime $\pl$, and on the description of its special fibre obtained via Cherednik-Drinfeld uniformisation \cite[Sections 1.4, 1.5]{nek12}.
\end{proof}

\subsection{Localisation of $c(\pl)$ at $\pl'\neq \pl$: the second reciprocity law.} The second ingredient in the Euler system argument we will use to prove Theorem \ref{mainthm} is a reciprocity law relating the localisation of $c(\pl)$ at an admissible prime $\pl'\neq \pl$ to a level raising of $g$ at the two primes $\pl, \pl'$.

\begin{teo}\label{serec}
Let $\pl'\neq \pl$ be an $n$-admissible prime not dividing $\mathfrak{D}_g$. Then
\begin{equation*}
loc_{\pl'}c(\pl)\in H^1_{ur}(K_{\pl'}, T_n(f))\simeq \Op/\varpi^n
\end{equation*}
and the following equality holds up to a $\varpi$-adic unit
\begin{equation*}
loc_{\pl'}c(\pl)= a(h)
\end{equation*}
where $h \in S^{B^{' \times}}(\Op/\varpi^n)$ is an admissible automorphic form on the quaternion algebra $B'$ of discriminant $\mathfrak{D}_g \pl \pl'$. Moreover the Hecke eigenvalues of $h$ for operators outside $\pl \pl '$ coincide with those of $g$, and $U_{\pl}h=\epsilon_\pl h$, $U_{\pl'}h=\epsilon_{\pl'}h$, where $\epsilon_{\pl} \in \{\pm 1 \}$ (resp. $\epsilon_{\pl'} \in \{\pm 1 \}$) is the number such that $N(\pl)+1 \equiv \epsilon_\pl \lambda_f(T_\pl) \pmod {\varpi^n}$ (resp. $N(\pl')+1 \equiv \epsilon_{\pl'} \lambda_f(T_{\pl'}) \pmod {\varpi^n}$).
\end{teo}
\begin{proof}
The proof goes as in \cite[Theorem 7.23]{lon12}, the only difference being that we are working with just one $CM$ point and not with a compatible tower of such. We point out that to construct the form $h$ one uses the fact that the supersingular locus in the special fibre of $X_\pl$ at $\pl'$ can be identified with $B'^\times\backslash \hat{B}^{' \times}/\hat{F}^\times\hat{R}^{'\times}$, where $R'\subset B'$ is a maximal order. Tha map $h$ is then constructed essentially taking the Abel-Jacobi image of points in the supersingular locus. The key point is to show that the map one obtaines is \emph{surjective}, which is \cite[Lemma 7.20]{lon12}. A different proof is given in \cite[Proposition 4.8]{liti17}; both proofs rely crucially on Ihara lemma for Shimura curves, which has been established over totally real fields in \cite{mash19} under our large image assumption (notice that the authors of \emph{loc. cit.} work with sufficiently small compact open subgroups throughout the paper, but remark in \cite[Remark 6.7]{mash19} that this assumption is not necessary in order for the result to hold).
\end{proof}

\begin{rem}\label{truerec}
The above theorem and the multiplicity one result in \cite{ma19} imply the following equality (up to unit), which will be used repeatedly later:
\begin{equation}\label{truerep}
loc_{\pl'}c(\pl)= loc_{\pl}c(\pl')
\end{equation}
for every couple of distinct admissible primes $\pl, \pl'$ not dividing $\mathfrak{D}_g$.
\end{rem}

\section{The Euler system argument}\label{esys}

\subsection{} We will now run the Euler system argument which proves Theorem \ref{mainthm}. We therefore start with an admissible form $g \in S^{B^\times}(\Op/\varpi^n)$ such that $n=2k$ and $a(g)\not \equiv 0 \pmod {\varpi^k}$. The main idea in the proof is to raise the level of $g$ at \emph{two} well-chosen admissible primes $\pl_1, \pl_2$, and construct an admissible automorphic form $h \in S^{B'^\times}(\Op/\varpi^n)$, where $B'$ has discriminant $\mathfrak{D}_h=\mathfrak{D}_g\pl_1\pl_2$, such that $ord_{\varpi}(a(h))<ord_{\varpi}(a(g))$ and we have
\begin{equation*}
2ord_{\varpi}(a(g_k))- 2ord_{\varpi}(a(h_k))= l_{\Op} Sel_{(\mathfrak{D}_g/\n)}(K, A_k(f))-l_{\Op} Sel_{(\mathfrak{D}_h/\n)}(K, A_k(f)).
\end{equation*}
One is thus reduced to prove the (in)equality in the case when $a(g)$ is a unit, which follows from the first reciprocity law.

\subsection{}\label{pastworks} This level raising-length lowering method is a refinement of ideas already used, for other purposes, by Wei Zhang in \cite{zha14} (as well as in \cite{pw11}). We will also make use in the first steps of our argument of a few lemmas essentially borrowed from \cite{how06}. A similar strategy is employed in the preprint \cite{blv16} as well: in particular a version of the important Lemma \ref{keylem} is also proved in \emph{loc. cit.} and used to perform an inductive process akin to ours. However in our argument we do not need to make use of parity results nor freeness results for certain Selmer groups which are employed in \cite{blv16}.

\subsection{} Let us first record a lemma which guarantees that there are sufficiently many admissible primes to detect whether a cohomology class in $H^1(K, A_1(f))$ is non zero. It was stated by Longo (\hspace{1sp}\cite[Theorem 4.3]{lon07}, \cite[Proposition 7.5]{lon12}) and proved by Wang \cite[Theorem 7.2]{wa15}. The proof relies on Chebotarev density theorem and the key point is that, in view of our large image assumption and \cite[Proposition 3.9]{dim05}, the image of the Galois representation $\bar{\rho}: \Gamma_F \rightarrow Aut(A_1(f))\simeq GL_2(\Op/\varpi)$ contains a matrix with eigenvalues $\lambda, \delta$ with $\lambda \neq \pm 1$ and $\delta \in \{\pm 1\}$ (notice that here we need that $p>3$).

\begin{lem}\label{locnonzero}
Let $c \in H^1(K, A_1(f))$ be a non zero class and $n \geq 1$. There are infinitely many $n$-admissible primes $\pl$ such that $loc_{\pl} (c) \neq 0$.
\end{lem}

\begin{corol}\label{lociso}
Let $C\subset Sel_{(\mathfrak{D}_g/\n)}(K, A_n(f))$ be a submodule isomorphic to $\Op/\varpi^n$. Then there exist infinitely many $n$-admissible primes $\pl$ such that $loc_{\pl}: Sel_{(\mathfrak{D}_g/\n)}(K, A_n(f))\rightarrow H^1_{ur}(K_\pl, A_n(f))$ is an isomorphism when restricted to $C$.
\end{corol}
\begin{proof}
Let us denote $A_n(f)$ by $M$; Let $c$ be a generator of $C$. Then $\varpi^{n-1}c \in Sel(K, M)[\varpi]=Sel(K, M[\varpi])$ is non zero, hence by Lemma \ref{locnonzero} there are infinitely many $n$-admissible primes $\pl$ (not dividing $\mathfrak{D}_g$) such that $loc_{\pl}(\varpi^{n-1}c)\neq 0$. For such a $\pl$ the localisation map $loc_{\pl}: C\rightarrow H^1_{ur}(K_\pl, M)\simeq \Op/\varpi^n$ is injective, hence an isomorphism.
\end{proof}

\subsection{} Let us now show that the first reciprocity law and the assumption that $a(g)$ does not vanish yield a weak annihilation result for the Selmer group. A similar result is proven in \cite[Proposition 2.3.5]{how06}.

\begin{prop}\label{weakannihil}
The $\Op/\varpi^n$-module $Sel_{(\mathfrak{D}_g/\n)}(K, A_n(f))$ is killed by $\varpi^{n-1}$; similarly, the $\Op/\varpi^k$-module $Sel_{(\mathfrak{D}_g/\n)}(K, A_k(f))$ is killed by $\varpi^{k-1}$.
\end{prop}

\begin{proof}
Let us prove the first statement; the second one is proved in the same way. Suppose by contradiction that there exists $c \in Sel_{(\mathfrak{D}_g/\n)}(K, A_n(f))$ which generates a submodule $C\simeq \Op/\varpi^n$. By Proposition \ref{lociso} we can choose $\pl\nmid \mathfrak{D}_g$ $n$-admissible such that $loc_{\pl}: C\rightarrow H^1_{ur}(K_\pl, A_n(f))\simeq \Op/\varpi^n$ is an isomorphism. In particular $loc_{\pl}: Sel_{(\mathfrak{D}_g/\n)}(K, A_n(f))\rightarrow H^1_{ur}(K_\pl, A_n(f))$ is surjective. We have two exact sequences:
\begin{align*}
0\rightarrow Sel_{(\mathfrak{D}_g/\n)\pl}(K, A_n(f))\rightarrow Sel_{(\mathfrak{D}_g/\n)}(K, A_n(f))\xrightarrow{loc_\pl}H^1_{ur}(K_{\pl}, A_n(f))\\
0\rightarrow Sel_{(\mathfrak{D}_g/\n)}(K, T_n(f))\rightarrow Sel_{(\mathfrak{D}_g/\n)}^{\pl}(K, T_n(f))\xrightarrow{loc_\pl}H^1_{tr}(K_\pl, T_n(f)).
\end{align*}
By global duality (\hspace{1sp}\cite[Theorem 2.3.4]{mr04}) the images of the two localisations maps are annihilators of each other. Since $loc_{\pl}: Sel_{(\mathfrak{D}_g/\n)}(K, A_n(f))\rightarrow H^1_{ur}(K_\pl, A_n(f))$ is surjective and the pairing
\begin{equation*}
H^1_{ur}(K_{\pl}, A_n(f))\times H^1_{tr}(K_\pl, T_n(f)) \rightarrow \Op/\varpi^n
\end{equation*}
is perfect we deduce that
\begin{equation*}
loc_{\pl}: Sel_{(\mathfrak{D}_g/\n)}^{\pl}(K, T_n(f))\rightarrow H^1_{tr}(K_\pl, T_n(f))
\end{equation*}
is the zero map. In particular $loc_{\pl}(c(\pl))=0$. But by the first reciprocity law $loc_\pl(c(\pl))=a(g)$ and $a(g)$ is non zero by hypothesis, which gives a contradiction.
\end{proof}

\begin{corol}(cf. \cite[Corollary 2.2.10, Remark 2.2.11]{how06})\label{changepar}
\begin{enumerate}
\item There exists an $\Op/\varpi^n$-module $N$ such that
\begin{equation*}
Sel_{(\mathfrak{D}_g/\n)}(K, A_n(f))=N\oplus N.
\end{equation*}
\item There exists an $\Op/\varpi^n$-module $N'$ such that
\begin{equation*}
Sel_{(\mathfrak{D}_g\pl/\n)}(K, A_n(f))=\Op/\varpi^n\oplus N'\oplus N'.
\end{equation*}
\end{enumerate}
\end{corol}
\begin{proof}
The first point follows immediately from the previous proposition and the structure theorem for Selmer groups \cite[Proposition 2.2.7]{how06}. In order to prove $(2)$, it is enough to show that the dimensions $\mathrm{dim}_{\Op/\varpi}Sel_{(\mathfrak{D}_g/\n)}(K, A_1(f))$ and $\mathrm{dim}_{\Op/\varpi}Sel_{(\mathfrak{D}_g\pl/\n)}(K, A_1(f))$ do not have the same parity. To prove this we argue as follows: we have two exact sequences
\begin{align*}
0\rightarrow Sel_{(\mathfrak{D}_g/\n)\pl}(K, A_1(f))\rightarrow Sel_{(\mathfrak{D}_g/\n)}(K, A_1(f))\xrightarrow{loc_\pl}H^1_{ur}(K_{\pl}, A_1(f))\\
0\rightarrow Sel_{(\mathfrak{D}_g/\n)}(K, A_1(f))\rightarrow Sel_{(\mathfrak{D}_g/\n)}^{\pl}(K, A_1(f))\xrightarrow{loc_\pl}H^1_{tr}(K_\pl, A_1(f)).
\end{align*}

By global duality if the upper localisation map is non zero then the bottom one is zero, hence in this case we obtain
\begin{equation*}
Sel_{(\mathfrak{D}_g/\n)}(K, A_1(f))=Sel_{(\mathfrak{D}_g/\n)}^{\pl}(K, A_1(f))
\end{equation*}
therefore
\begin{equation*}
Sel_{(\mathfrak{D}_g\pl/\n)}(K, A_1(f))=Sel_{(\mathfrak{D}_g/\n)\pl}(K, A_1(f)).
\end{equation*}
Hence
\begin{align*}
\mathrm{dim}Sel_{(\mathfrak{D}_g/\n)}(K, A_1(f))-\mathrm{dim}Sel_{(\mathfrak{D}_g\pl/\n)}(K, A_1(f))&=\\
\mathrm{dim}Sel_{(\mathfrak{D}_g/\n)}(K, A_1(f))-\mathrm{dim}Sel_{(\mathfrak{D}_g/\n)\pl}(K, A_1(f))&=1.
\end{align*}
If the upper localisation map is zero then the bottom one is non zero and one argues similarly.
\end{proof}

\begin{prop}(cf. \cite[Lemma 3.3.6]{how06})\label{freesys}
There exists a free $\Op/\varpi^k$-submodule of rank one of $Sel_{(\mathfrak{D}_g\pl/\n)}(K, T_k(f))$ which contains (the reduction modulo $\varpi^k$ of) $c(\pl)$.
\end{prop}
\begin{proof}
By the previous corollary we can write
\begin{align*}
Sel_{(\mathfrak{D}_g\pl/\n)}(K, T_n(f))= & \Op/\varpi^n \oplus N\oplus N \\
Sel_{(\mathfrak{D}_g\pl/\n)}(K, T_k(f))= & \Op/\varpi^k \oplus M\oplus M. 
\end{align*}
We know that $c(\pl)$ is non zero (modulo $\varpi^k$), since this is true for its localisation at $\pl$. We claim that this implies that $\varpi^{k-1}M=0$. If this is not the case, then $Sel_{(\mathfrak{D}_g\pl/\n)}(K, T_k(g))$ contains a free $\Op/\varpi^k$-submodule of rank 2, hence, for any admissible prime $\pl' \neq \pl$, the kernel $ Sel_{(\mathfrak{D}_g\pl\pl'/\n)}(K, T_k(g))$ of the localisation map
\begin{equation*}
loc_{\pl'}:Sel_{(\mathfrak{D}_g\pl/\n)}^{\pl'}(K, T_k(g))\rightarrow H^1_{ur}(K_{\pl'}, T_k(g))
\end{equation*}
contains a (non zero) free $\Op/\varpi^k$-submodule. Therefore, writing $Sel_{(\mathfrak{D}_g\pl\pl'/\n)}(K, T_k(g))=P\oplus P$, we have $\varpi^{k-1}P\neq 0$. With the same argument as in the proof of Proposition \ref{weakannihil} we deduce that $a(h)=0$, where $h$ is a level raising modulo $\varpi^k$ of $g$ at $\pl\pl'$. The second reciprocity law yields $loc_{\pl'}c(\pl)=a(h)=0$; since this is true for every admissible prime $\pl'$ we get $c(\pl)=0$, which gives a contradiction.

Hence $\varpi^{k-1}M=0$. We have a commutative diagram
\begin{center}
\begin{tikzcd}
Sel_{(\mathfrak{D}_g\pl/\n)}(K, T_n(f)) \arrow[r, "\cdot \varpi^{k}"] \arrow[d] & Sel_{(\mathfrak{D}_g\pl/\n)}(K, T_n(f))[\varpi^k]\\
Sel_{(\mathfrak{D}_g\pl/\n)}(K, T_k(f)) \arrow[ru]
\end{tikzcd}
\end{center}
where the diagonal arrow is an isomorphism. Since $\varpi^{k-1}M=0$ we deduce that the $\Op/\varpi^k$-module $\varpi^{k-1}Sel_{(\mathfrak{D}_g\pl/\n)}(K, T_k(f))$ is cyclic, hence the same holds for $\varpi^{k-1}Sel_{(\mathfrak{D}_g\pl/\n)}(K, T_n(f))[\varpi^k]$. Therefore $\varpi^{k-1}N=0$, so $N$ is killed by the horizontal map. This implies that the image of the vertical arrow is free of rank one; since it contains the reduction modulo $\varpi^k$ of $c(\pl)$, the proof is complete.
\end{proof}

\begin{rem}\label{freeness}
The above property, which will play an important role in the proof of the sought-for estimate for the length of the Selmer group, explains why we need to work with the reduction modulo $\varpi^k$ of automorphic forms modulo $\varpi^{2k}$. Let us say (after Howard \cite[Definition 2.3.6]{how06}, whose proof of a very similar result we closely followed) that our Euler system is free if it enjoys the property in the above proposition. Then the Euler system modulo $\varpi^{2k}$ may not be free, but its reduction modulo $\varpi^k$ is.
\end{rem}

\subsection{} Let us set $t(g_k)=ord_\varpi(a(g_k))$ and $t(g_k, \pl)=ord_\varpi(c(\pl))$ for $\pl$ an $n$-admissible prime (not dividing $\mathfrak{D}_g$). We are seeing $c(\pl)$ as an element of $Sel_{(\mathfrak{D}_g\pl/\n)}(K, T_k(f))$, and we remark that $ord_\varpi(c(\pl))$ can be calculated in any submodule $C\simeq \Op/\varpi^k$ containing $c(\pl)$, whose existence is guaranteed by the previous proposition. Indeed, let $c(\pl)=\varpi^a u$, where $u \in C$ is a unit. Then clearly $a$ is smaller than the order of $c(\pl)$ in $Sel_{(\mathfrak{D}_g\pl/\n)}(K, T_k(f))$. We claim that equality holds. Indeed, suppose that there exists $v \in Sel_{(\mathfrak{D}_g\pl/\n)}(K, T_k(f))$ and $b>a$ such that $\varpi^{b}v=c(\pl)$. Then we have $\varpi^{b}v=\varpi^a u$, hence
\begin{equation*}
\varpi^{k-1} u=\varpi^{a+k-a-1} u=\varpi^{b+k-a-1}v
\end{equation*}
The left hand side is non zero, but the right hand side is zero, since $b+k-a-1\geq 1+k-1=k$; this yields a contradiction and proves our claim.

\subsection{} We have the following chain of inequalities:
\begin{equation*}
t(g_k, \pl)=ord_\varpi(c(\pl))\leq ord_\varpi(loc_\pl c(\pl))=t(g_k)<k
\end{equation*}

where the last equality follows from the first reciprocity law, and the last inequality holds because of our assumption that $a(g)\not\equiv 0 \pmod {\varpi^k}$. Hence there exists a class $\kappa(\pl) \in Sel_{(\mathfrak{D}_g\pl/\n)}(K, T_k(f))$ such that $c(\pl)=\varpi^{t(g_k, \pl)}\kappa(\pl)$. Our previous discussion implies that the class $\kappa(\pl)$ can (and will) be taken to be in a submodule $C\simeq \Op/\varpi^k$. It enjoys the following properties:
\begin{enumerate}
\item $\kappa(\pl) \in Sel_{(\mathfrak{D}_g\pl/\n)}(K, T_k(f))$.
\item $ord_{\varpi} \kappa(\pl)=0$.
\item $ord_\varpi(loc_\pl (\kappa(\pl)))=t(g_k)-t(g_k, \pl)$.
\end{enumerate}

\begin{lem}\label{lemred}
Suppose that $Sel_{(\mathfrak{D}_g/\n)}(K, A_k(f))\neq 0$. Then there exist infinitely many admissible primes $\pl \nmid \mathfrak{D}_g$ such that $t(g_k, \pl)<t(g_k)$.
\end{lem}

\begin{proof}
Let $c \in Sel_{(\mathfrak{D}_g/\n)}(K, A_k(f))$ be a non zero class, and $\pl$ an admissible prime not dividing $\mathfrak{D}_g$ such that $loc_{\pl}(c) \neq 0$. By global duality and the fact that the local conditions defining $Sel_{(\mathfrak{D}_g\pl/\n)}(K, T_k(f))$ and $Sel_{(\mathfrak{D}_g/\n)}(K, A_k(f))$ are everywhere orthogonal except at $\pl$ we have:
\begin{equation*}
0=\sum_v\langle loc_v (c), loc_v (\kappa(\pl)) \rangle=\langle loc_\pl(c), loc_\pl (\kappa(\pl)) \rangle.
\end{equation*}
Since the pairing between $H^1_{ur}(K_\pl, A_k(f))$ and $H^1_{tr}(K_\pl, T_k(f))$ is perfect and $loc_{\pl}(c) \neq 0$ we deduce that $loc_\pl (\kappa(\pl)) \in \Op/\varpi^k$ cannot be a unit. By property $(3)$ above, this proves the lemma.
\end{proof}

\begin{corol}\label{ineq0}
If $a(g_k)$ is a unit then $Sel_{(\mathfrak{D}_g/\n)}(K, A_k(f))=0$.
\end{corol}
\begin{proof}
This follows immediately from the previous lemma. We remark that this can also be deduced from Proposition \ref{weakannihil}, replacing $A_k(f)$ with $A_1(f)$ and using the hypothesis that $a(g_k)$ is not congruent to $0$ modulo $\varpi$. These two proofs essentially rely on the same argument.
\end{proof}

\subsection{} Recall that we want to prove Theorem \ref{mainthm}. We will prove it by induction on $t(g_k)=ord_\varpi(a(g_k))$, which is finite by assumption. The above corollary deals with the base case $t(g_k)=0$; to treat the general case we will make use of the following

\begin{lem}\label{keylem}
Suppose that $Sel_{(\mathfrak{D}_g/\n)}(K, A_k(f))$ is non zero. Then there exist two $n$-admissible primes $\pl_1\neq \pl_2$ not dividing $\mathfrak{D}_g$ and an admissible automorphic form $h \in S^{B^{'^\times}}(\Op/\varpi^n)$, where $B'$ is the totally definite quaternion algebra of discriminant $\mathfrak{D}_g\pl_1\pl_2$, such that:
\begin{enumerate}
\item $t(g_k, \pl_1)=t(g_k, \pl_2)<t(g_k)$.
\item $t(h_k)=t(g_k, \pl_i)$, $i=1,2$.
\item $ord_\varpi loc_{\pl_1}(\kappa(\pl_2))=ord_\varpi loc_{\pl_2}(\kappa(\pl_1))=0$.
\item $Sel_{(\mathfrak{D}_g \pl_1 \pl_2/\n)}(K, A_k(f))=Sel_{(\mathfrak{D}_g/\n)\pl_1\pl_2}(K, A_k(f))$.
\end{enumerate}
\end{lem}
\begin{proof}
Take $\pl_1$ admissible such that $t(g_k,\pl_1)=\mathrm{min}\{t(g_k,\pl), \pl \; n-\text{admissible prime}\}$. Lemma \ref{lemred} and the assumption that the Selmer group is non trivial imply that $t(g_k,\pl_1) < t(g_k)$. We know that $ord_\varpi(\kappa(\pl_1))=0$ and that $\kappa(\pl_1)\in C \subset Sel_{(\mathfrak{D}_g\pl_1/\n)}(K, T_k(f))$, where $C\simeq \Op/\varpi^k$. Hence by Corollary \ref{lociso} (which holds true modulo $\varpi^k$) we can choose an $n$-admissible prime $\pl_2$ distinct from $\pl_1$ and not dividing $\mathfrak{D}_g$ such that
\begin{equation*}
ord_\varpi(loc_{\pl_2}(\kappa(\pl_1)))=0.
\end{equation*}
We now have the following chain of equalities:
\begin{align*}
t(g_k,\pl_1)+ord_\varpi(loc_{\pl_2}(\kappa(\pl_1)))= & ord_\varpi( loc_{\pl_2}(c(\pl_1)))\\
&= t(h_k)=ord_\varpi(loc_{\pl_1}(c(\pl_2)))\\
&= t(g_k,\pl_2)+ord_\varpi(loc_{\pl_1}(\kappa(\pl_2)))
\end{align*}
where the second and third equalities follow from the second reciprocity law in the form given in Remark \ref{truerec}.\\
Now
\begin{equation*}
t(g_k,\pl_1)\leq t(g_k, \pl_2)
\end{equation*}
by minimality of $t(g_k,\pl_1)$, and $ord_\varpi(loc_{\pl_2}(\kappa(\pl_1)))=0$. Comparing the first, third and last member in the chain of equalities above we deduce that
\begin{align*}
t(g_k,\pl_1) = t(g_k, \pl_2)&= t(h_k),\\
ord_\varpi(loc_{\pl_1}(\kappa(\pl_2)))&= 0.
\end{align*}
Hence claims $(1), (2)$ and $(3)$ are proved.\\
It remains to show $(4)$. We have two exact sequences:
\begin{align*}
Sel_{(\mathfrak{D}_g\pl_1\pl_2/\n)}(K, T_k(f)) \hookrightarrow Sel_{(\mathfrak{D}_g/\n)}^{\pl_1 \pl_2}(K, T_k(f))\xrightarrow{v_{\pl_1}\oplus v_{\pl_2}} H^1_{ur}(K_{\pl_1}, T_k(f))\oplus H^1_{ur}(K_{\pl_2}, T_k(f))\\
Sel_{(\mathfrak{D}_g/\n)\pl_1 \pl_2}(K, A_k(f)) \hookrightarrow Sel_{(\mathfrak{D}_g\pl_1\pl_2/\n)}(K, A_k(f))\xrightarrow{\delta_{\pl_1}\oplus \delta_{\pl_2}} H^1_{tr}(K_{\pl_1}, A_k(f))\oplus H^1_{tr}(K_{\pl_2}, A_k(f))
\end{align*}
where $v_{\pl_i}$ (resp. $\delta_{\pl_i}$) denotes the composition of the localisation map and the projection onto the unramified (resp. transverse) part.

By Poitou-Tate global duality the images of $v_{\pl_1}\oplus v_{\pl_2}$ and $\delta_{\pl_1}\oplus \delta_{\pl_2}$ are orthogonal complements with respect to the local Tate pairing. Now, the classes $\kappa(\pl_1)$ and $\kappa(\pl_2)$ belong to $Sel_{(\mathfrak{D}_g/\n)}^{\pl_1 \pl_2}(K, T_k(f))$, and because of $(3)$ and the fact that the localisation at $\pl_i$ of $\kappa(\pl_i)$ falls in the transverse part we have, up to unit:
\begin{align*}
v_{\pl_1}\oplus v_{\pl_2}(\kappa(\pl_1))=&(0,1)\\
v_{\pl_1}\oplus v_{\pl_2}(\kappa(\pl_2))=&(1,0).
\end{align*}
This implies that the map
\begin{equation*}
v_{\pl_1}\oplus v_{\pl_2}:Sel_{(\mathfrak{D}_g/\n)}^{\pl_1 \pl_2}(K, T_k(f))\rightarrow H^1_{ur}(K_{\pl_1}, T_k(f))\oplus H^1_{ur}(K_{\pl_2}, T_k(f))
\end{equation*}
is surjective. Since the pairing between $H^1_{ur}(K_{\pl_i}, T_k(f))$ and $H^1_{tr}(K_{\pl_i}, A_k(f))$ is perfect for $i=1,2$ we deduce that 
\begin{equation*}
\delta_{\pl_1}\oplus \delta_{\pl_2}: Sel_{(\mathfrak{D}_g\pl_1\pl_2/\n)}(K, A_k(f))\longrightarrow H^1_{tr}(K_{\pl_1}, A_k(f))\oplus H^1_{tr}(K_{\pl_2}, A_k(f))
\end{equation*}
is the zero map, therefore we have an isomorphism:
\begin{equation*}
Sel_{(\mathfrak{D}_g/\n)\pl_1\pl_2}(K, A_k(f)) \simeq Sel_{(\mathfrak{D}_g\pl_1\pl_2/\n)}(K, A_k(f)).
\end{equation*}
\end{proof}

\subsection{} Let us now prove the inequality
\begin{equation*}
l_{\Op}Sel_{(\mathfrak{D}_g/\n)}(K, A_k(f))\leq 2 t(g_k)
\end{equation*}
by induction on $t(g_k)$. If $t(g_k)=0$ then the inequality follows from Corollary \ref{ineq0}, hence let us suppose that $t(g_k)>0$. If $Sel_{(\mathfrak{D}_g/\n)}(K, A_k(f))$ is trivial then there is nothing to prove. If $Sel_{(\mathfrak{D}_g/\n)}(K, A_k(f))$ is non trivial, choose two $n$-admissible primes $\pl_1$, $\pl_2$ as in Lemma \ref{keylem}.

We have two exact sequences:
\begin{align*}
Sel_{(\mathfrak{D}_g/\n)}(K, T_k(f)) \hookrightarrow Sel_{(\mathfrak{D}_g/\n)}^{\pl_1 \pl_2}(K, T_k(f))\xrightarrow{\delta_{\pl_1}\oplus \delta_{\pl_2}} H^1_{tr}(K_{\pl_1}, T_k(f))\oplus H^1_{tr}(K_{\pl_2}, T_k(f))\\
Sel_{(\mathfrak{D}_g/\n)\pl_1 \pl_2}(K, A_k(f)) \hookrightarrow Sel_{(\mathfrak{D}_g/\n)}(K, A_k(f))\xrightarrow{v_{\pl_1}\oplus v_{\pl_2}} H^1_{ur}(K_{\pl_1}, A_k(f))\oplus H^1_{ur}(K_{\pl_2}, A_k(f)).
\end{align*}

Let us identify $H^1_{tr}(K_{\pl_i}, T_k(f))$ with $H^1_{ur}(K_{\pl_i}, A_k(f))^\vee=\mathrm{Hom}_{\Op/\varpi^k}(H^1_{ur}(K_{\pl_i}, A_k(f)), \Op/\varpi^k)$ via the local Tate pairing at $\pl_i$, for $i=1, 2$. Taking the dual of the lower exact sequence above and using Poitou-Tate global duality we find an exact sequence:
\begin{align*}
Sel_{(\mathfrak{D}_g/\n)}(K, T_k(f)) & \hookrightarrow Sel_{(\mathfrak{D}_g/\n)}^{\pl_1 \pl_2}(K, T_k(f))\xrightarrow{\delta_{\pl_1}\oplus \delta_{\pl_2}} H^1_{tr}(K_{\pl_1}, T_k(f))\oplus H^1_{tr}(K_{\pl_2}, T_k(f))\\ &\xrightarrow{v_{\pl_1}^\vee\oplus v_{\pl_2}^\vee} Sel_{(\mathfrak{D}_g/\n)}(K, A_k(f))^\vee \longrightarrow Sel_{(\mathfrak{D}_g/\n)\pl_1 \pl_2}(K, A_k(f))^\vee \longrightarrow 0.
\end{align*}

Using $(4)$ of Lemma \ref{keylem} we deduce that:
\begin{align}\label{lgtsel}
\nonumber l_{\Op}Sel_{(\mathfrak{D}_g/\n)}(K, A_k(f))-l_{\Op}Sel_{(\mathfrak{D}_g\pl_1\pl_2/\n)}(K, A_k(f))=&\\
l_{\Op}Sel_{(\mathfrak{D}_g/\n)}(K, T_k(f))-l_{\Op}Sel_{(\mathfrak{D}_g/\n)}^{\pl_1 \pl_2}(K, T_k(f))+2k.
\end{align}

Let us now compute $l_{\Op}Sel^{\pl_1 \pl_2}_{(\mathfrak{D}_g/\n)}(K, T_k(f))-l_{\Op}Sel_{(\mathfrak{D}_g/\n)}(K, T_k(f))$. Choose an element $\zeta(\pl_1) \in Sel_{(\mathfrak{D}_g/\n)}^{\pl_1}(K, T_k(f))$ such that $\delta_{\pl_1}(\zeta(\pl_1))$ generates the image of the map
\begin{equation*}
Sel_{(\mathfrak{D}_g/\pl)}^{\pl_1}(K, T_k(f))\xrightarrow{\delta_{\pl_1}}H^1_{tr}(K_{\pl_1}, T_k(f))\simeq \Op/\varpi^k.
\end{equation*}

We find an exact sequence:
\begin{equation}\label{exseqselr}
0\longrightarrow Sel_{(\mathfrak{D}_g/\n)}(K, T_k(f))\longrightarrow Sel_{(\mathfrak{D}_g/\n)}^{\pl_1}(K, T_k(f))\xrightarrow{\delta_{\pl_1}} \delta_{\pl_1}(\zeta(\pl_1))\Op/\varpi^k\longrightarrow 0.
\end{equation}

The cohomology class $\kappa(\pl_1)$ belongs to $Sel_{(\mathfrak{D}_g/\n)}^{\pl_1}(K, T_k(f))$; hence, possibly after multiplying it by a unit of $\Op/\varpi^k$, there exists an integer $m_1 \geq 0$ such that
\begin{equation*}
\delta_{\pl_1}(\varpi^{m_1}\zeta(\pl_1)-\kappa(\pl_1))=0.
\end{equation*}
This implies:
\begin{equation*}
m_1+ord_\varpi(\delta_{\pl_1}(\zeta(\pl_1)))=ord_\varpi(\delta_{\pl_1}(\kappa(\pl_1)))=t(g_k)-t(g_k, \pl_1)=t(g_k)-t(h_k)
\end{equation*}
where the third equality follows from (2) of Lemma \ref{keylem}.
Using this and the exact sequence \eqref{exseqselr} we obtain:
\begin{align*}
l_{\Op}Sel_{(\mathfrak{D}_g/\n)}^{\pl_1}(K, T_k(f))-l_{\Op}Sel_{(\mathfrak{D}_g/\n)}(K, T_k(f))&= k-ord_\varpi(\delta_{\pl_1}(\zeta(\pl_1)))\\
&= k+ m_1-t(g_k)+t(h_k).
\end{align*}

Similarly, take $\zeta(\pl_2) \in Sel_{(\mathfrak{D}_g/\n)}^{\pl_1\pl_2}(K, T_k(f))$ such that we have an exact sequence:
\begin{equation*}
0\longrightarrow Sel^{\pl_1}_{(\mathfrak{D}_g/\n)}(K, T_k(f))\longrightarrow Sel^{\pl_1\pl_2}_{(\mathfrak{D}_g/\n)}(K, T_k(f))\xrightarrow{\delta_{\pl_2}} \delta_{\pl_2}(\zeta(\pl_2))\Op/\varpi^k\longrightarrow 0.
\end{equation*}
Then there exists $m_2\geq 0$ such that $\delta_{\pl_2}(\varpi^{m_2}\zeta(\pl_2)-\kappa(\pl_2))=0$, hence we find:
\begin{align*}
l_{\Op}Sel_{(\mathfrak{D}_g/\n)}^{\pl_1\pl_2}(K, T_k(f))-l_{\Op}Sel_{(\mathfrak{D}_g/\n)}^{\pl_1}(K, T_k(f))&= k-ord_\varpi(\delta_{\pl_2}(\zeta(\pl_2)))\\
&= k+ m_2 - t(g_k)+t(h_k).
\end{align*}
Therefore we obtain:
\begin{equation*}
l_{\Op}Sel_{(\mathfrak{D}_g/\n)}^{\pl_1\pl_2}(K, T_k(f))-l_{\Op}Sel_{(\mathfrak{D}_g/\n)}(K, T_k(f))= 2k+ m_1+ m_2-2t(g_k)+2t(h_k).
\end{equation*}
This, together with equation \eqref{lgtsel}, yields:
\begin{equation*}
l_{\Op}Sel_{(\mathfrak{D}_g/\n)}(K, A_k(f))-l_{\Op}Sel_{(\mathfrak{D}_g\pl_1\pl_2/\n)}(K, A_k(f))=-m_1-m_2+2t(g_k)-2t(h_k)
\end{equation*}
which finally implies:
\begin{equation}\label{keyeq}
l_{\Op}Sel_{(\mathfrak{D}_g/\n)}(K, A_k(f))-2t(g_k)=l_{\Op}Sel_{(\mathfrak{D}_g\pl_1\pl_2/\n)}(K, A_k(f))-2t(h_k)-m_1-m_2.
\end{equation}
Since $t(h_k)< t(g_k)$, by induction we have $l_{\Op}Sel_{(\mathfrak{D}_g\pl_1\pl_2/\n)}(K, A_k(f))-2t(h_k)\leq 0$, hence by \eqref{keyeq} we also have $l_{\Op}Sel_{(\mathfrak{D}_g/\n)}(K, A_k(f))-2t(g_k)\leq 0$.

\subsection{}
We have completed the proof of the inequality in the statement of Theorem \ref{mainthm}. It remains to prove that the equality also holds, under the additional hypothesis that the implication
\begin{equation*}
Sel_{(\mathfrak{D}_h/\n)}(K, A_1(f))=0\Longrightarrow t(h)=0
\end{equation*}
holds true for every admissible automorphic form $h$ modulo $\varpi$.\\
As before, the proof is by induction on $t(g_k)$, and the case $t(g_k)=0$ is covered by Lemma \ref{ineq0}. So let us suppose that $t(g_k)>0$. Then, \emph{since we are assuming that}
\begin{equation}\label{assimp}
Sel_{(\mathfrak{D}_g/\n)}(K, A_1(f))=0\Longrightarrow t(g_1)=0
\end{equation}
we deduce that $Sel_{(\mathfrak{D}_g/\n)}(K, A_1(f))$ \emph{cannot be trivial}, hence we can invoke Lemma \ref{keylem}. Let us stress, before continuing the proof, that it is at this point that the proof of the equality differs substantially from the proof of the inequality we gave above, and the non trivial input \eqref{assimp} is crucially needed.

Let $\pl_1$ and $\pl_2$ be two admissible primes as in Lemma \ref{keylem}, and let $h$ be the automorphic form given by the lemma. We proved above the following equality \eqref{keyeq}:
\begin{equation*}
l_{\Op}Sel_{(\mathfrak{D}_g/\n)}(K, A_k(f))-2t(g_k)=l_{\Op}Sel_{(\mathfrak{D}_g\pl_1\pl_2/\n)}(K, A_k(f))-2t(h_k)-m_1-m_2
\end{equation*}
Moreover we know that $t(h_k) < t(g_k)$. Therefore by induction we have 
\begin{equation*}
l_{\Op}Sel_{(\mathfrak{D}_g\pl_1\pl_2/\n)}(K, A_k(f))=2t(h_k).
\end{equation*}
In order to complete the proof it is therefore enough to show that $m_1=m_2=0$.

\subsection{} Let us first show that $m_1=0$. Recall that $m_1$ was chosen in such a way that the equality $\delta_{\pl_1}(\varpi^{m_1}\zeta(\pl_1)-\kappa(\pl_1))=0$ is satisfied. In other words, the class $\varpi^{m_1}\zeta(\pl_1)-\kappa(\pl_1)$, which a priori lives in $Sel_{(\mathfrak{D}_g/\n)}^{\pl_1}(K, T_k(f))$, actually belongs to $Sel_{(\mathfrak{D}_g/\n)}(K, T_k(f))$. Lemma \ref{weakannihil} yields the equality:
\begin{equation*}
\varpi^{k-1}\kappa(\pl_1)=\varpi^{m_1+k-1}\zeta(\pl_1)
\end{equation*}
hence:
\begin{equation*}
\varpi^{k-1}loc_{\pl_2}(\kappa(\pl_1))=\varpi^{m_1+k-1}loc_{\pl_2}(\zeta(\pl_1)) \in H^1_{ur}(K, T_k(f))\simeq \Op/\varpi^k.
\end{equation*}
By Lemma \ref{keylem} we have $ord_\varpi(loc_{\pl_2}(\kappa(\pl_1)))=0$, hence the left hand side of the above equality is non zero. Therefore the right hand side must also be non trivial, yielding  $m_1+k-1 < k$. Hence $m_1=0$.

\subsection{} Let us finally show that $m_2=0$. Since we already know that $m_1=0$ we have $\delta_{\pl_1}(\zeta(\pl_1))=\delta_{\pl_1}(\kappa(\pl_1))$ (up to unit). By definition of $\zeta(\pl_1)$, this implies that $\delta_{\pl_1}(\kappa(\pl_1))$ generates the image of the map $Sel_{(\mathfrak{D}_g/\n)}^{\pl_1}(K, T_k(f))\xrightarrow{\delta_{\pl_1}}H^1_{tr}(K_{\pl_1}, T_k(f))$.\\
Now recall that $m_2$ was chosen so that $\delta_{\pl_2}(\varpi^{m_2}\zeta(\pl_2)-\kappa(\pl_2))=0$, which implies that
\begin{equation*}
\varpi^{m_2}\zeta(\pl_2)-\kappa(\pl_2) \in Sel_{(\mathfrak{D}_g/\n)}^{\pl_1}(K, T_k(f))\subset Sel_{(\mathfrak{D}_g/\n)}^{\pl_1\pl_2}(K, T_k(f)).
\end{equation*}
Therefore there exists $m_3 \geq 0$ such that:
\begin{equation*}
\delta_{\pl_1}(\varpi^{m_2}\zeta(\pl_2)-\kappa(\pl_2)-\varpi^{m_3}\kappa(\pl_1))=0.
\end{equation*}
In other words, we have $\varpi^{m_2}\zeta(\pl_2)-\kappa(\pl_2)-\varpi^{m_3}\kappa(\pl_1) \in Sel_{(\mathfrak{D}_g/\n)}(K, T_k(f))$. Invoking Lemma \ref{weakannihil} again we obtain
\begin{equation*}
\varpi^{m_2+k-1}\zeta(\pl_2)-\varpi^{k-1}\kappa(\pl_2)=\varpi^{m_3+k-1}\kappa(\pl_1)
\end{equation*}
hence:
\begin{equation*}
loc_{\pl_1}(\varpi^{m_2+k-1}\zeta(\pl_2))-loc_{\pl_1}(\varpi^{k-1}\kappa(\pl_2))=loc_{\pl_1}(\varpi^{m_3+k-1}\kappa(\pl_1)).
\end{equation*}
Suppose by contradiction that $m_2>0$. Then the first term in the above equation dies, and we get:
\begin{equation*}
-loc_{\pl_1}(\varpi^{k-1}\kappa(\pl_2))=loc_{\pl_1}(\varpi^{m_3+k-1}\kappa(\pl_1)).
\end{equation*}
Notice that
\begin{align*}
loc_{\pl_1}(\varpi^{k-1}\kappa(\pl_2)) \in H^1_{ur}(K_{\pl_1}, T_k(f))\\
loc_{\pl_1}(\varpi^{m_3+k-1}\kappa(\pl_1))\in H^1_{tr}(K_{\pl_1}, T_k(f)).
\end{align*}
Hence both terms must be zero. On the other hand, since $ord_\varpi(loc_{\pl_1}(\kappa(\pl_2)))=0$ the left hand side of the above equality is non trivial. This gives a contradiction, and completes the proof of Theorem \ref{mainthm}.

\section{The indefinite case}

\subsection{}\label{defindefclass} Let us now switch to the \emph{indefinite setting}, namely we fix a Hilbert newform $f \in S(\n)$ and a CM extension $K/F$ such that $\n$ is squarefree and all its factors are inert in $K$, and we assume in addition that $[F: \Q]\not \equiv \#\{\q, \q \mid \n\} \pmod 2$. In this case the sign of the functional equation of $L(f_K, s)$ is $-1$, and Zhang's special value formula asserts that $L'(f_K, 1)=\frac{2^{r+1}}{\sqrt{N(disc(K/F))}}\cdot \langle f, f \rangle_{Pet}\cdot \langle a(f), a(f)\rangle_{NT}$. Recall (see section \ref{specvalind}) that $a(f) \in (Jac(X)(K)\otimes\Op)/I_{f_B}$ is the $f_B$-isotypical part of the trace of a point $P_K$ with $CM$ by $\mathcal{O}_K$ on the quotient Shimura curve $X$ with full level structure attached to the quaternion algebra $B$ of discriminant $\n$ ramified at all but one infinite place. Furthermore in this setting the $f_B$-isotypical part of the Tate module $T_p(Jac(X))$ is isomorphic to $T(f)$ as a $\Gamma_F$-module, as a consequence of the Eichler-Shimura relations for Shimura curves. It follows that $a(f)$ gives rise to a cohomology class $c \in Sel(K, T(f))$. Our aim in this section is to prove that, if $L'(f_K, 1)\neq 0$, then $Sel(K, A(f))$ has $\Op$-corank one and we have the inequality (which is an equality under the same additional assumption as in Theorem \ref{mainthm})
\begin{equation}\label{mainind}
l_{\Op}Sel(K, A(f))/div \leq 2ord_\varpi(c),
\end{equation}
where we denote by $Sel(K, A(f))/div$ the quotient of $Sel(K, A(f))$ by its maximal divisible submodule.

\begin{rem}
In the simplest case when the Hecke eigenvalues of $f$ are rational we have $V(f)=T_p(E_f)\otimes_{\Z_p}\Q_p$, where $T_p(E_f)$ is the $p$-adic Tate module of an elliptic curve $E_f/F$ with $L$-function $L(f, s)$. The above inequality then translates into a relation between (the $p$-parts of) the cardinality of the Tate-Shafarevich group of $E/K$ and the square of the index of the Heegner point in $E_f(K)$ coming from $a(f)$; this is consistent with what predicted by the Birch and Swinnerton-Dyer conjecture (see \cite[Lemma 10.1.2]{zha14}).
\end{rem}

We are going to prove the following result:

\begin{teo}\label{mainmodnindef}
Fix $f \in S(\n)$ and $K/F$ satisfying assumptions $(1), (2), (3)$ of Theorem \ref{thbkdef}. Assume that $L'(f_K, 1)\neq 0$. Let $n=2k$, and suppose that $c \not \equiv 0 \pmod {\varpi^k}$. Then the following inequality holds:
\begin{equation*}
l_{\Op}Sel(K, A_k(f))\leq k +2ord_\varpi(c).
\end{equation*}
Moreover the above inequality is an equality provided that the following implication holds true: if $h$ is an admissible automorphic form mod $\varpi$ and $Sel_{(\mathfrak{D}_h/\n)}(K, A_1(f))=0$ then $0\neq a(h) \in \Op/\varpi$.
\end{teo}

\subsection{} Theorem \ref{mainmodnindef} implies the inequality \eqref{mainind}. Indeed, assume that the theorem holds and write $Sel(K, A(f))=(E_\p/\Op)^r\oplus M$ with $M$ finite. Then $l_{\Op}Sel(K, A_k(f))=kr+\lOp M[\varpi^k]$, hence $r=1$ and for $k$ large enough we have
\begin{equation*}
l_{\Op}(M)=l_{\Op}(M[\varpi^k])=l_{\Op} Sel(K, A_k(f))-k \leq 2ord_\varpi(c),
\end{equation*}
hence $l_{\Op}Sel(K, A(f))/div \leq 2ord_\varpi(c)$, as we had to show.

In order to prove Theorem \ref{mainmodnindef} we will make use of the second reciprocity law \ref{serec}, relating the localisation of the class $c$ at an admissible prime $\pl$ to the quantity $a(g)$, where $g \in S^{B^{'\times}}(\Op/\varpi^n)$ is an admissible automorphic form on the totally definite quaternion algebra $B'$ of discriminant $\n \pl$. This brings us in the context of Theorem \ref{mainthm}, and choosing $\pl$ appropriately we will deduce Theorem \ref{mainmodnindef} from Theorem \ref{mainthm}. We wish to stress the remarkable fact that the second reciprocity law allows to prove special value formulas in analytic rank one by reducing them to the rank zero case.

\subsection{} Let us prove Theorem \ref{mainmodnindef}.
Let $t(f)=ord_\varpi(c)$. The reduction modulo $\varpi^k$ of $c$ is contained in a free $\Op/\varpi^k$-module $C$ of rank one. This is proved in the same way as in Proposition \ref{freesys}, once one we know that
\begin{equation*}
Sel(K, T_n(f))=\Op/\varpi^n\oplus N \oplus N.
\end{equation*}
To show this, choose $\pl$ admissible such that $loc_{\pl}(c)\neq 0 \in \Op/\varpi^n$. By the second reciprocity law this implies that $a(g)\neq 0$, where $g$ is a level raising of $f$ at $\pl$ modulo $\varpi^n$. Hence Proposition \ref{weakannihil} applies, and it implies that $Sel_{(\pl)}(K, T_n(f))\simeq M\oplus M$. By Corollary \ref{changepar} we conclude.

There exists a class $\kappa \in C\subset Sel(K, T_k(f))$ such that $\varpi^{t(f)}\kappa=c$.
Hence we can choose an admissible prime $\pl$ such that $ord_\varpi(loc_\pl(\kappa))=0$. Using the second reciprocity law we find:
\begin{equation}\label{defindeq}
t(f)=ord_\varpi(loc_\pl(c))=ord_\varpi(a(g_k))
\end{equation}
where $g \in S^{B^{'\times}}(\Op/\varpi^n)$ is an admissible automorphic form on the totally definite quaternion algebra $B'$ of discriminant $\n \pl$.

\subsection{} To prove Theorem \ref{mainmodnindef} we shall now compare the Selmer groups $Sel(K, A_k(f))$ and $Sel_{(\pl)}(K, A_k(f))$.

We have a square of Selmer groups:
\begin{center}
\begin{tikzcd}
& Sel^{\pl}(K,A_k(f))&\\
Sel(K, A_k(f))\arrow{ur}{c} & & Sel_{(\pl)}(K, A_k(f))\arrow{ul}{d} \\
& Sel_{\pl}(K, A_k(f))\arrow{ul}{a}\arrow{ur}{b}&
\end{tikzcd}
\end{center}

Global duality yields an exact sequence:
\begin{align}\label{globdu1}
0\longrightarrow Sel_{(\pl)}(K, T_k(f))\longrightarrow & Sel^\pl(K, T_k(f))\xrightarrow{v_\pl}H^1_{ur}(K_\pl, T_k(f))\\
\nonumber \xrightarrow{\delta_\pl^\vee}&Sel_{(\pl)}(K, A_k(f))^\vee \longrightarrow Sel_\pl(K, A_k(f))^\vee\longrightarrow 0.
\end{align}

Since $\kappa\in Sel^\pl(K, T_k(f))$ satisfies $ord_\varpi(loc_\pl(\kappa))=0$ the map $v_\pl$ is surjective, therefore $\delta_\pl^\vee$ is the zero map, which yields:
\begin{equation*}
Sel_{(\pl)}(K, A_k(f))\simeq Sel_\pl(K, A_k(f)).
\end{equation*}
In other words, the inclusion $b$ in the square above is an isomorphism. This implies that

\begin{align}\label{bdlength}
l_{\Op} Sel^{\pl}(K, A_k(f))-l_{\Op} Sel_{\pl}(K, A_k(f))=\\
\nonumber l_{\Op} Sel^{\pl}(K, A_k(f))-l_{\Op} Sel_{(\pl)}(K, A_k(f))\leq k.
\end{align}

Since the class $\kappa\in Sel(K, T_k(f))\simeq Sel(K, A_k(f))$ satisfies $ord_\varpi(loc_\pl(\kappa))=0$, we find an exact sequence
\begin{equation*}
0 \longrightarrow Sel_\pl(K, A_k(f)) \longrightarrow Sel(K, A_k(f))\xrightarrow{v_\pl}H^1_{ur}(K_\pl, A_k(f))\longrightarrow 0
\end{equation*}
which yields
\begin{equation*}
l_{\Op} Sel(K, A_k(f))-l_{\Op} Sel_{\pl}(K, A_k(f))=k.
\end{equation*}

Because of \eqref{bdlength} we see that the map $c$ is an isomorphism. Collecting everything we get
\begin{equation*}
l_{\Op}Sel_{(\pl)}(K, A_k(f))=l_{\Op}Sel_\pl(K, A_k(f))=l_{\Op}Sel(K, A_k(f))-k.
\end{equation*}

Now $g$ is an admissible automorphic form satisfying the hypotheses of Theorem \ref{mainthm}, hence

\begin{equation}\label{hkineq}
l_{\Op}Sel_{(\pl)}(K, A_k(f))\leq 2ord_\varpi(a(g_k)).
\end{equation}

Finally, using equation \eqref{defindeq} we obtain

\begin{equation*}
l_{\Op}Sel(K, A_k(f))\leq 2t(f)+k
\end{equation*}

and equality holds whenever it does in equation \eqref{hkineq}. Hence the proof is complete.

\subsection{A remark on parity of the dimension of Selmer groups.} Let us conclude by mentioning the implications that the level raising-length lowering method we used has for parity results for Selmer groups. These results are already known in our setting \cite{nek06}; we hope that the following argument  - inspired to \cite[Section 9.2]{zha14}; see also \cite[Lemma 9]{gp12} - can be of interest nonetheless. It allows to prove that the \emph{parity conjecture}, asserting that the parity of the $\Op$-corank of $Sel(K, A(f))$ equals the parity of the order of vanishing of $L(f_K, s)$ at $s=1$, follows from (hence is equivalent to) the \emph{sign conjecture}, which predicts that $Sel(K, A(f))\neq 0$ whenever $\epsilon(f_K)=-1$. Work related to the latter conjecture has been carried out in \cite{su06}, \cite{beche09}, \cite{bel12}.\\
Let $f \in S(\n)$ and let $K/F$ be a $CM$ extension such that the assumptions $(1), (2), (3)$ of Theorem \ref{thbkdef} are satisfied. We wish to prove that
\begin{equation*}
\mathrm{dim}_{\Op/\varpi} Sel(K, A_1(f))\equiv \epsilon(f_K) \pmod 2
\end{equation*}
by induction on $d(f)=\mathrm{dim}_{\Op/\varpi} Sel(K, A_1(f))$, and assuming that the congruence holds true whenever $d(f)=0$. Suppose that $Sel(K, A_1(f))\neq 0$ and choose a non zero class $c \in Sel(K, A_1(f))$ as well as an admissible prime $\pl$ such that $loc_\pl(c) \neq 0$. Then $loc_{\pl}:Sel(K, A_1(f))\rightarrow H^1_{ur}(K_\pl, A_1(f))$ is surjective, hence by global duality we obtain that $Sel(K, A_1(f))=Sel^\pl(K, A_1(f))$ and $Sel_{(\pl)}(K, A_1(f))=Sel_{\pl}(K, A_1(f))$. Hence $\mathrm{dim}_{\Op/\varpi} Sel_{(\pl)}(K, A_1(f))=\mathrm{dim}_{\Op/\varpi} Sel(K, A_1(f))-1$. On the other hand $Sel_{(\pl)}(K, A_1(f))$ is the mod $\varpi$ Selmer group of a level raising $g \in S(\n\pl)$ of $f$ (see \ref{changesel}; notice that in this argument one needs to work with modular forms in characteristic zero). By induction we have $\epsilon(g_K)\equiv \mathrm{dim}_{\Op/\varpi}Sel_{(\pl)}(K, A_1(f)) \pmod 2$. Finally, $\epsilon(g_K)=-\epsilon(f_K)$ as the numbers of prime ideals dividing $\n$ and $\n\pl$ have different parity.

\bibliographystyle{amsalpha}
\bibliography{mybib}

\providecommand{\bysame}{\leavevmode\hbox to3em{\hrulefill}\thinspace}
\providecommand{\MR}{\relax\ifhmode\unskip\space\fi MR }
\providecommand{\MRhref}[2]{%
  \href{http://www.ams.org/mathscinet-getitem?mr=#1}{#2}
}
\providecommand{\href}[2]{#2}
\begin{thebibliography}{BBV16}

\bibitem[BBV16]{bbv16}
Andrea Berti, Massimo Bertolini, and Rodolfo Venerucci, \emph{Congruences
  between modular forms and the {B}irch and {S}winnerton-{D}yer conjecture},
  Elliptic curves, modular forms and {I}wasawa theory, Springer Proc. Math.
  Stat., vol. 188, Springer, Cham, 2016, pp.~1--31.

\bibitem[BC09]{beche09}
Jo\"{e}l Bella\"{\i}che and Ga\"{e}tan Chenevier, \emph{Families of {G}alois
  representations and {S}elmer groups}, Ast\'{e}risque (2009), no.~324,
  xii+314.

\bibitem[BD05]{bd05}
Massimo Bertolini and Henri Darmon, \emph{Iwasawa's main conjecture for
  elliptic curves over anticyclotomic {$\Bbb Z_p$}-extensions}, Ann. of Math.
  (2) \textbf{162} (2005), no.~1, 1--64.

\bibitem[Bel12]{bel12}
Jo\"{e}l Bella\"{\i}che, \emph{Ranks of {S}elmer groups in an analytic family},
  Trans. Amer. Math. Soc. \textbf{364} (2012), no.~9, 4735--4761.

\bibitem[BFH90]{bufriho90}
Daniel Bump, Solomon Friedberg, and Jeffrey Hoffstein, \emph{Nonvanishing
  theorems for {$L$}-functions of modular forms and their derivatives}, Invent.
  Math. \textbf{102} (1990), no.~3, 543--618.

\bibitem[BK90]{bk90}
Spencer Bloch and Kazuya Kato, \emph{{$L$}-functions and {T}amagawa numbers of
  motives}, The {G}rothendieck {F}estschrift, {V}ol. {I}, Progr. Math.,
  vol.~86, Birkh\"{a}user Boston, Boston, MA, 1990, pp.~333--400.

\bibitem[BLV19]{blv16}
Massimo Bertolini, Matteo Longo, and Rodolfo Venerucci, \emph{The
  anticyclotomic main conjectures for elliptic curves}, preprint.

\bibitem[BR89]{blaro89}
Don Blasius and Jonathan Rogawski, \emph{Galois representations for {H}ilbert
  modular forms}, Bull. Amer. Math. Soc. (N.S.) \textbf{21} (1989), no.~1,
  65--69.

\bibitem[BR93]{blro93}
\bysame, \emph{Motives for {H}ilbert modular forms}, Invent. Math. \textbf{114}
  (1993), no.~1, 55--87.

\bibitem[CH15]{chi15}
Masataka Chida and Ming-Lun Hsieh, \emph{On the anticyclotomic {I}wasawa main
  conjecture for modular forms}, Compos. Math. \textbf{151} (2015), no.~5,
  863--897.

\bibitem[Chi17]{chi17}
Masataka Chida, \emph{Selmer groups and central values of {$L$}-functions for
  modular forms}, Ann. Inst. Fourier (Grenoble) \textbf{67} (2017), no.~3,
  1231--1276.

\bibitem[CST14]{cashti14}
Li~Cai, Jie Shu, and Ye~Tian, \emph{Explicit {G}ross-{Z}agier and {W}aldspurger
  formulae}, Algebra Number Theory \textbf{8} (2014), no.~10, 2523--2572.

\bibitem[DI08]{darjo08}
Henri Darmon and Adrian Iovita, \emph{The anticyclotomic main conjecture for
  elliptic curves at supersingular primes}, J. Inst. Math. Jussieu \textbf{7}
  (2008), no.~2, 291--325.

\bibitem[Dim05]{dim05}
Mladen Dimitrov, \emph{Galois representations modulo $p$ and cohomology of
  {H}ilbert modular varieties}, Annales scientifiques de l'\'Ecole Normale
  Sup\'erieure \textbf{Ser. 4, 38} (2005), no.~4, 505--551.

\bibitem[DS74]{dese74}
Pierre Deligne and Jean-Pierre Serre, \emph{Formes modulaires de poids {$1$}},
  Ann. Sci. \'{E}cole Norm. Sup. (4) \textbf{7} (1974), 507--530 (1975).
  \MR{0379379}

\bibitem[Fon92]{fon92}
Jean-Marc Fontaine, \emph{Valeurs sp\'{e}ciales des fonctions {$L$} des
  motifs}, no. 206, 1992, S\'{e}minaire Bourbaki, Vol. 1991/92, pp.~Exp. No.
  751, 4, 205--249.

\bibitem[FPR94]{fpr94}
Jean-Marc Fontaine and Bernadette Perrin-Riou, \emph{Autour des conjectures de
  {B}loch et {K}ato: cohomologie galoisienne et valeurs de fonctions {$L$}},
  Motives ({S}eattle, {WA}, 1991), Proc. Sympos. Pure Math., vol.~55, Amer.
  Math. Soc., Providence, RI, 1994, pp.~599--706.

\bibitem[GP12]{gp12}
Benedict~H. Gross and James~A. Parson, \emph{On the local divisibility of
  {H}eegner points}, Number theory, analysis and geometry, Springer, New York,
  2012, pp.~215--241.

\bibitem[How06]{how06}
Benjamin Howard, \emph{Bipartite {E}uler systems}, J. Reine Angew. Math.
  \textbf{597} (2006), 1--25.

\bibitem[JSW17]{jeskiwa17}
Dimitar Jetchev, Christopher Skinner, and Xin Wan, \emph{The {B}irch and
  {S}winnerton-{D}yer formula for elliptic curves of analytic rank one}, Camb.
  J. Math. \textbf{5} (2017), no.~3, 369--434.

\bibitem[KL91]{kolo91}
V.~A. Kolyvagin and D.~Yu. Logach\"{e}v, \emph{Finiteness of {CH} over totally
  real fields}, Izv. Akad. Nauk SSSR Ser. Mat. \textbf{55} (1991), no.~4,
  851--876.

\bibitem[Liu19]{liu19}
Yifeng Liu, \emph{Bounding cubic-triple product {S}elmer groups of elliptic
  curves}, J. Eur. Math. Soc. (JEMS) \textbf{21} (2019), no.~5, 1411--1508.

\bibitem[Lon06]{lon06}
Matteo Longo, \emph{On the {B}irch and {S}winnerton-{D}yer conjecture for
  modular elliptic curves over totally real fields}, Ann. Inst. Fourier
  (Grenoble) \textbf{56} (2006), no.~3, 689--733.

\bibitem[Lon07]{lon07}
\bysame, \emph{Euler systems obtained from congruences between {H}ilbert
  modular forms}, Rend. Semin. Mat. Univ. Padova \textbf{118} (2007), 1--34.

\bibitem[Lon12]{lon12}
\bysame, \emph{Anticyclotomic {I}wasawa's main conjecture for {H}ilbert modular
  forms}, Comment. Math. Helv. \textbf{87} (2012), no.~2, 303--353.

\bibitem[LT17]{liti17}
Yifeng Liu and Yichao Tian, \emph{Supersingular locus of {H}ilbert modular
  varieties, arithmetic level raising and selmer groups}, preprint (2017).

\bibitem[LV10]{lv10}
Matteo Longo and Stefano Vigni, \emph{On the vanishing of {S}elmer groups for
  elliptic curves over ring class fields}, J. Number Theory \textbf{130}
  (2010), no.~1, 128--163.

\bibitem[Man19]{ma19}
Jeffrey Manning, \emph{Patching and multiplicity {$2^k$} for {S}himura curves},
  arXiv:1902.06878 (2019).

\bibitem[MR04]{mr04}
Barry Mazur and Karl Rubin, \emph{Kolyvagin systems}, Mem. Amer. Math. Soc.
  \textbf{168} (2004), no.~799, viii+96.

\bibitem[MS19]{mash19}
Jeffrey Manning and Jack Shotton, \emph{Ihara's lemma for {S}himura curves over
  totally real fields via patching}, arXiv:1907.06043 (2019).

\bibitem[Nek06]{nek06}
Jan Nekov\'{a}\v{r}, \emph{Selmer complexes}, Ast\'erisque (2006), no.~310,
  viii+559.

\bibitem[Nek07]{nek07}
\bysame, \emph{The {E}uler system method for {CM} points on {S}himura curves},
  London Mathematical Society Lecture Note Series, p.~471–547, Cambridge
  University Press, 2007.

\bibitem[Nek12]{nek12}
\bysame, \emph{Level raising and anticyclotomic {S}elmer groups for {H}ilbert
  modular forms of weight two}, Canad. J. Math. \textbf{64} (2012), no.~3,
  588--668.

\bibitem[Oht82]{oht82}
Masami Ohta, \emph{On {$l$}-adic representations attached to automorphic
  forms}, Japan. J. Math. (N.S.) \textbf{8} (1982), no.~1, 1--47.

\bibitem[PW11]{pw11}
Robert Pollack and Tom Weston, \emph{On anticyclotomic {$\mu$}-invariants of
  modular forms}, Compos. Math. \textbf{147} (2011), no.~5, 1353--1381.

\bibitem[Ray74]{ra74}
Michel Raynaud, \emph{Sch\'emas en groupes de type {$(p,\dots, p)$}}, Bull.
  Soc. Math. France \textbf{102} (1974), 241--280.

\bibitem[Rib76]{rib76}
Kenneth~A. Ribet, \emph{A modular construction of unramified {$p$}-extensions
  of {$\mathbf{Q}(\mu _{p})$}}, Invent. Math. \textbf{34} (1976), no.~3,
  151--162.

\bibitem[SU06]{su06}
Christopher Skinner and Eric Urban, \emph{Vanishing of {$L$}-functions and
  ranks of {S}elmer groups}, International {C}ongress of {M}athematicians.
  {V}ol. {II}, Eur. Math. Soc., Z\"{u}rich, 2006, pp.~473--500.

\bibitem[SU14]{su14}
\bysame, \emph{The {I}wasawa main conjectures for {$\rm GL_2$}}, Invent. Math.
  \textbf{195} (2014), no.~1, 1--277.

\bibitem[Tay89]{tay89}
Richard Taylor, \emph{On {G}alois representations associated to {H}ilbert
  modular forms}, Invent. Math. \textbf{98} (1989), no.~2, 265--280.

\bibitem[Tay95]{tay95}
\bysame, \emph{On {G}alois representations associated to {H}ilbert modular
  forms. {II}}, Elliptic curves, modular forms, \& {F}ermat's last theorem
  ({H}ong {K}ong, 1993), Ser. Number Theory, I, Int. Press, Cambridge, MA,
  1995, pp.~185--191.

\bibitem[Vat03]{vat03}
Vinayak Vatsal, \emph{Special values of anticyclotomic {$L$}-functions}, Duke
  Math. J. \textbf{116} (2003), no.~2, 219--261.

\bibitem[Wan14]{wan14}
Xin Wan, \emph{Iwasawa main conjecture for supersingular elliptic curves},
  arXiv:1411.6352 (2014).

\bibitem[Wan15]{wan15}
\bysame, \emph{The {I}wasawa main conjecture for {H}ilbert modular forms},
  Forum Math. Sigma \textbf{3} (2015), e18, 95.

\bibitem[Wan19]{wa15}
Haining Wang, \emph{Anticyclotomic {I}wasawa theory for {H}ilbert modular
  forms}, arxiv:1909.12374 (2019).

\bibitem[Wil88]{wi88}
Andrew Wiles, \emph{On ordinary {$\lambda$}-adic representations associated to
  modular forms}, Invent. Math. \textbf{94} (1988), no.~3, 529--573.

\bibitem[YZZ13]{yzz13}
Xinyi Yuan, Shou-Wu Zhang, and Wei Zhang, \emph{The {G}ross-{Z}agier formula on
  {S}himura curves}, Princeton University Press, 2013.

\bibitem[Zha04]{zh04}
Shou-Wu Zhang, \emph{Gross-{Z}agier formula for {$\rm GL(2)$}. {II}}, Heegner
  points and {R}ankin {$L$}-series, Math. Sci. Res. Inst. Publ., vol.~49,
  Cambridge Univ. Press, Cambridge, 2004, pp.~191--214.

\bibitem[Zha14]{zha14}
Wei Zhang, \emph{Selmer groups and the indivisibility of {H}eegner points},
  Camb. J. Math. \textbf{2} (2014), no.~2, 191--253.

\bibitem[Zho19]{zho19}
Rong Zhou, \emph{Motivic cohomology of quaternionic {S}himura varieties and
  level raising}, arXiv:1901.01954 (2019).

\end{thebibliography}

\Addr

\end{document}